\documentclass[12pt]{article}
 \usepackage{amsmath,amsthm,amssymb,amstext,graphicx,psfrag}
 \usepackage{hyperref}
\newtheorem{theorem}{Theorem}

\newtheorem{lemma}{Lemma}

\newtheorem{proposition}{Proposition}
\newtheorem{example}{Example}

\newtheorem{algorithm}{Algorithm}
\theoremstyle{definition}
\newtheorem{definition}{Definition}
\newtheorem{remark}{Remark}
\newtheorem{remarks}{Remarks}

\def\sp{\mathop{\rm sp}}

\def\Id{{\rm Id}}

\newcommand{\R}{\mathbb{R}}
\newcommand{\Z}{\mathbb{Z}}
\newcommand{\N}{\mathbb{N}}

\newcommand{\argmin}{\mathop{\mathrm{argmin}}}
\newcommand{\argmax}{\mathop{\mathrm{argmax}}}

\newcommand{\set}[1]{\left\{#1\right\}}

\renewcommand{\mod}[1]{\ (\mathrm{mod}\ #1)}

\title{Coherent sets for nonautonomous dynamical systems}
\author{Gary Froyland, Simon Lloyd,
and Naratip Santitissadeekorn \\ School of Mathematics and Statistics\\ University of New South Wales \\ Sydney, NSW 2052, Australia}

\begin{document}
\maketitle
\begin{abstract}
We describe a mathematical formalism and numerical algorithms for
identifying and tracking slowly mixing objects in nonautonomous
dynamical systems.  In the autonomous setting, such objects are
variously known as almost-invariant sets, metastable sets,
persistent patterns, or strange eigenmodes, and have proved to be
important in a variety of applications.  In this current work, we
explain how to extend existing autonomous approaches to the
nonautonomous setting. We call the new time-dependent slowly mixing
objects \emph{coherent sets} as they represent regions of phase
space that disperse very slowly and remain coherent.  The new
methods are illustrated via detailed examples in both discrete and
continuous time.
\end{abstract}

%\tableofcontents

\section{Introduction}

The study of transport and mixing in dynamical systems has received
considerable attention in the last two decades; see e.g.\
\cite{Meiss1992,wiggins_92,aref_02,wiggins_05} for discussions of
transport phenomena. In particular, the detection of very slowly
mixing objects, known variously as almost-invariant sets, metastable
sets, persistent patterns, or strange eigenmodes, has found wide
application in fields such as fluid dynamics
\cite{pikovsky_popovych_03,liu_haller_04,popovych_pikovsky_eckhardt_07},
ocean dynamics \cite{froyland_padberg_england_treguier_07,
npg}, astrodynamics
\cite{dellnitz_etal_05}, and molecular dynamics
\cite{deuflhard_etal_00,schuette_huisinga_deuflhard_01}.  A
shortcoming of this prior work, based around eigenfunctions of
Perron--Frobenius operators (or transfer operators, or evolution
operators) is the restriction to autonomous systems or periodically
forced systems. In this work, we extend the notions of
almost-invariant sets, metastable sets, persistent patterns, and
strange eigenmodes to time-dependent
 \emph{Lagrangian coherent sets}.  These coherent sets form a time parameterised family of sets that approximately follow
the flow and disperse very slowly; in other words \textit{they stay coherent}.  Coherent sets are the natural nonautonomous analogue to
almost-invariant sets.

The standard dynamical systems model of transport assumes that the
motion of passive particles are completely determined by either an
autonomous or a time-dependent vector field. Traditional approaches
to understanding transport are based upon the determination of the
location of geometric objects such as invariant manifolds.  In the
autonomous setting, an invariant manifold of one dimension less than
the ambient space will form an impenetrable transport barrier that
locally partitions the ambient space. In the periodically-forced
setting, primarily in two-dimensional flows, it has been shown that
slow mixing in the neighbourhood of invariant manifolds is sometimes
controlled by ``lobe dynamics''
\cite{romkedar_etal_1990,romkedar_wiggins_90,wiggins_92}.  In the
truly non-autonomous, or aperiodically forced setting, finite-time
hyperbolic material lines \cite{haller_00} and surfaces
\cite{haller_01} have been proposed as generalisations of invariant
manifolds that form barriers to mixing.  These material lines and
surfaces are known as \emph{Lagrangian coherent structures};  see
also \cite{shadden_lekien_marsden_05} for an alternative definition.
The geometric approach can often be used to find co-dimension 1 sets
(coherent structures) that form boundaries of coherent sets.

An alternative to the geometric approach is the ergodic theoretic
approach, which attempts to locate almost-invariant sets (or
metastable sets) directly, rather than inferring their location
indirectly from their boundaries. The basic tool is the
Perron--Frobenius operator (or transfer operator).  Real eigenvalues
of this operator close to 1 correspond to eigenmodes that decay at
slow (exponential) rates.
%These eigenmodes are given by the
%eigenfunctions $f$ of the operator corresponding to the large real
%eigenvalues.
Almost-invariant sets are heuristically
determined from the corresponding eigenfunctions $f$ as sets of the form $\{f> c\}$
or $\{f<c\}$ for thresholds $c\in\mathbb{R}$. Such an approach arose
in the context of smooth autonomous maps and flows on subsets of
$\mathbb{R}^d$ \cite{dellnitz_junge_99, dellnitz_junge_97} about a
decade ago. Further theoretical and computational extensions have
since been constructed \cite{froyland_dellnitz_03, froyland_05,
froyland_08}. A parallel series of work specific to time-symmetric
Markov processes and applied to identifying molecular conformations
was developed in
\cite{schuettehabil,deuflhard_etal_98,deuflhard_etal_00} and
surveyed in \cite{schuette_huisinga_deuflhard_01}.

There have been some recent studies of the connections between slow
mixing in \emph{periodically} driven fluid flow and eigenfunctions
of Perron--Frobenius operators. Liu and Haller \cite{liu_haller_04}
observe via simulation a transient ``strange eigenmode'' as
predicted by classical Floquet theory. Pikovsky and Popovych
\cite{pikovsky_popovych_03,popovych_pikovsky_eckhardt_07}
 numerically integrated an advection-diffusion equation to simulate the evolution
of a passive scalar, observing that it is the sub-dominant
eigenfunction of the Perron--Frobenius operator that describes the
most persistent deviation from the unique steady state.

The Perron--Frobenius operator based approach has been successful in
a variety of application areas, however, as the key mathematical
object is an \emph{eigen}function, there is no simple extension of
the method to systems that have \emph{nonperiodic} time
dependence\footnote{A relevant analogy to see this is the following.
Consider repeated application of a single matrix $A$.  The
eigenvectors of $A$ provide information on directions of exponential
growth/decay specified by the corresponding eigenvalues. Similarly,
the eigenvectors of a product of matrices $A_k\cdots A_2 A_1$
describe directions of exponential growth/decay, specified by the
eigenvalues of the product, under \emph{repeated} application of
this matrix product.  However, the directions of exponential
growth/decay under a \emph{non-repeating} product $\cdots A_k\cdots
A_2 A_1$ cannot be in general be found as eigenvectors of some
matrix.}

%Related ideas have also been developed for finite-state Markov
%chains \cite{gaveau_schulman_98,gaveau_schulman_06}, where the
%starting point is a Markov chain model of some physical system that
%is similar in spirit to a transfer operator.

%The particular
%form of flow used in
%\cite{pikovsky_popovych_03,popovych_pikovsky_eckhardt_07} admitted a
%convenient Fourier series representation that allowed calculation of
%leading eigenmodes. The numerical methods we describe in the present
%paper can be used to estimate eigenmodes for very general flows and
%are computationally attractive.

Indeed, Liu and Haller \cite{liu_haller_04} state that:
\begin{quote}
``...strange eigenmodes may also be viewed as eigenfunctions of an
appropriate Frobenius-Perron operator...This fresh approach offers
an alternative view on scalar mixing, but leaves the questions of
completeness and general time-dependence open.''
\end{quote}
It is this question of general time-dependence that we address in
the current work. We extend a standard formalism for random
dynamical systems to the level of Perron--Frobenius operators to
create a Perron--Frobenius operator framework for general
time-dependence. We also state an accompanying numerical algorithm,
and demonstrate its effectiveness in identifying strange eigenmodes
and coherent sets.

An outline of the paper is as follows. In Section 2 we formalise the
notions of nonautonomous systems in both discrete and continuous
time.  In Section 3 we describe a Galerkin projection method that we
will use to produce finite matrix representations of
Perron--Frobenius operators.  In Section 4 we define the critical
constructions for the nonautonomous setting.  We show that the nonautonomous analogues of strange eigenmodes are described by the ``Oseledets subspaces'' or
``Lyapunov vectors'' corresponding to compositions of the projected
Perron--Frobenius operators.  In Section 5 we describe in detail a
numerical algorithm to practically compute these slowly decaying
modes, and demonstrate that in the continuous time setting, these
modes vary continuously in time.  Our numerical approach is
illustrated firstly in the discrete time setting with an aperiodic
composition of interval maps, and secondly in the continuous time
setting with an aperiodically forced flow on a cylinder.  Section 6
provides some further background on almost-invariant sets and
coherent sets and Section 7 describes a new heuristic to extract
coherent sets from slowly decaying modes in the nonautonomous
setting.  This heuristic is then illustrated using the examples from
Section 5.
\section{Nonautonomous Dynamical Systems}

We will treat time dependent dynamical systems on a smooth compact
$d$-dimensional manifold $M\subset\mathbb{R}^D$, $D\ge d$ in both
discrete and continuous time. In order to keep track of ``time'' we
use a probability space $(\Omega,\mathcal{H},\mathbb{P})$, with the
passing of time controlled by an ergodic automorphism
$\theta:\Omega\circlearrowleft$ preserving $\mathbb{P}$ (ie.\
$\mathbb{P}=\mathbb{P}\circ \theta^{-t}$ for all $t\ge 0$). We
require this somewhat more complicated description of time for
technical reasons: to run the ergodic-theoretic arguments in Theorem
\ref{thm:main} that guarantee the existence of the nonautonomous
analogues of strange eigenmodes. The requirement that $\mathbb{P}$
be an ergodic probability measure rules out obvious choices for
$\Omega$ and $\theta$:  (i) in discrete time, $\Omega=\mathbb{Z}$
and $\theta^s(t)=t+s$, and (ii) in continuous time,
$\Omega=\mathbb{R}$ and $\theta^s(t)=t+s$. In both (i) and (ii),
there is no ergodic probability measure on $\Omega$ preserved by
$\theta$. In the next two sections, we will introduce suitable
examples of $\Omega$ and $\theta$ and describe the nonautonomous
systems they generate.

%However, as already mentioned at the beginning of this section, in
%order to do ergodic theory we require a driving system leaving some
%probability measure $\mathbb{P}$ invariant. This is not possible if
%$\Omega=\mathbb{R}$ and $\theta^s(t)=t+s$ is a translation by $s$.

\subsection{Discrete time -- Maps}

In the discrete time setting, we will think of $\Omega\subset
(\mathbb{Z})^\mathbb{Z}$, and $\theta$ as a left shift $\sigma$ on
$\Omega$ defined by $(\sigma\omega)_i=\omega_{i+1}$, where
$\omega=(\ldots,\omega_{-1},\omega_0,\omega_1,\ldots)\in\Omega$. We
assume that $\sigma$ is ergodic with respect to $\mathbb{P}$. Let
$\mathcal{T}=\{T_{\omega_0}\}_{{\omega_0}\in \mathbb{Z}}$ be a
collection of (possibly non-invertible) piecewise differentiable maps on a
compact manifold $M$. For brevity, we will sometimes write
$T_\omega$ in place of $T_{\omega_0}$. We will define a
nonautonomous dynamical system by map compositions of the form
$T_{\sigma^{k-1}\omega}\circ\cdots\circ T_{\sigma\omega}\circ
T_{\omega}$. Define
$$\Phi(k,\omega,x):=\left\{
                      \begin{array}{ll}
                        T_{\sigma^{k-1}\omega}\circ\cdots\circ
T_{\sigma\omega}\circ T_{\omega}(x), & \hbox{$k>0$;} \\
                        \Id, & \hbox{$k=0$;} \\
                        T^{-1}_{\sigma^{-k}\omega}\circ\cdots\circ
T^{-1}_{\sigma^{-2}\omega}\circ T^{-1}_{\sigma^{-1}\omega}(x), & \hbox{$k<0$.}
                      \end{array}
                    \right.
$$
For $k\ge 0$ (resp.\ $k<0$), $\Phi(k,\omega,x)$ represents the forward time (resp.\ backward time) $k$-fold application
of the nonautonomous dynamics to the point $x$ initialised at
``time'' $\omega$.  Whenever $T_\omega$ is non-invertible, $T^{-1}_\omega(x)$ will represent the finite set of all preimages of $x$. We call $\Phi$ a \emph{map cocycle}.

\begin{definition}
\label{invariantmeasurediscrete} Endow $M$ with the Borel
$\sigma$-algebra and let $\mu$ be a probability measure on $M$.  We
call $\mu$ an \emph{invariant measure} if
$\mu\circ\Phi(-1,\omega,\cdot)=\mu$ for all $\omega\in\Omega$.
\end{definition}

This definition of an invariant measure is stricter than is usual
for random or nonautonomous dynamical systems (e.g.\  \cite[Definition 1.4.1]{arnoldbook}).  More generally, one may allow
sample measures $\mu=\mu_\omega$ and insist that
$\mu_{\sigma^{-1}\omega}\circ\Phi(-1,\omega,\cdot)=\mu_{\omega}$ for
all $\omega\in\Omega$.

\begin{example}[Aperiodic map cocycle]\label{Discreteeg}
We construct a map cocycle $\Phi$ by the composition of maps $T_i$ from a collection $\mathcal{T}$ according to sequences of indices $\omega\in\Omega$.
The collection $\mathcal{T}:=\set{T_1,T_2,T_3,T_4}$ consists of expanding maps of the circle $S^1$, which we think of as $[0,1]$ with endpoints identified. The sequence space $\Omega\subset\set{1,2,3,4}^\Z$ is given by
$$
\Omega = \set{\omega\in\set{1,2,3,4}^\Z:\forall i\in\Z,\, M_{\omega_i\omega_{i+1}}=1},
$$
with adjacency matrix
$$
M=\left(\begin{array}{cccc}
1&1&0&0\\
0&0&1&1\\
1&1&0&0\\
0&0&1&1
\end{array}\right).
$$
Elements of $\Omega$ correspond to bi-infinite paths in the graph Figure 1.
\begin{figure}[tb]
    \centering
    \psfrag{1}{\small{1}}
    \psfrag{2}{\small{2}}
    \psfrag{3}{\small{3}}
    \psfrag{4}{\small{4}}
        \includegraphics[width=0.60\textwidth]{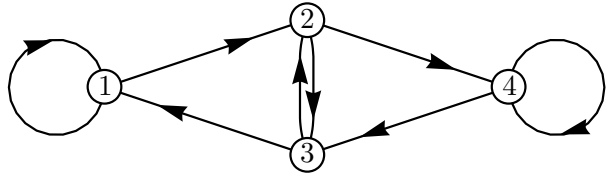}
    \caption{Graph of the sequence space $\Omega$.}
    \label{fig:subshift}
\end{figure}
The shift $\sigma:\Omega\to\Omega$ is a subshift of finite type. A
Borel $\sigma$-algebra $\mathcal{H}$ is generated by the length-one
cylinder sets $C_i=\set{\omega:\omega_0=i}$, $i=1,\ldots,4$, and by
giving equal measure to these four cylinder sets, we generate a
shift-invariant probability measure $\mathbb{P}$.

The maps of $\mathcal{T}$ are defined in terms of a continuous piecewise-linear map $H_a:S^1\to S^1$, which has almost-invariant sets (see Definition \ref{aidefn}, Section 6) $[0,0.5]$ and $[0.5,1]$ for $a$ close to zero.
Define
$$
H_a(x) = \left\{\begin{array}{ll}
+3x         & 0 \leq x < \frac{1}{6} + \frac{1}{2}a, \\
-3x+3a+1& \frac{1}{6} + \frac{1}{2}a \leq x < \frac{1}{3} + \frac{2}{3}a, \\
+3x-a-1 &  \frac{1}{3} + \frac{2}{3}a \leq x < \frac{2}{3} + \frac{2}{3}a, \\
-3x+3a+3& \frac{2}{3} + \frac{2}{3}a \leq x < \frac{5}{6} + \frac{1}{2}a, \\
+3x-2   &  \frac{5}{6} + \frac{1}{2}a \leq x \leq 1,
\end{array}\right.
$$
where values are taken modulo $1$. Figure \ref{fig:H0} shows a graph of $H_{0}$.
\begin{figure}[tb]
    \centering
    \psfrag{T1}{$H_{0}$}
        \includegraphics[width=0.30\textwidth]{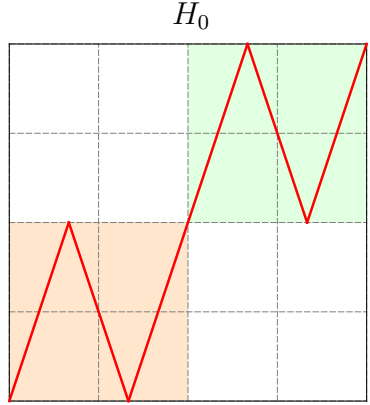}
    \caption{The map $H_{0}$ has invariant sets $[0,0.5]$ and $[0.5,1]$;  that is, $H_0^{-1}([0,0.5])=[0,0.5]$ and $H_0^{-1}([0.5,1])=[0.5,1]$.}
    \label{fig:H0}
\end{figure}
Let $a_i\in\R$, $i=1,\ldots,4$, be close to zero, for example $(a_1,a_2,a_3,a_4)=(\pi,2\sqrt{2}, \sqrt{3},e)/40$. We now construct the map $T_i$ from $H_{a_i}$, for $i=1,\ldots,4$ as follows:
\begin{eqnarray*}
T_1 &=& H_{a_1}(x) \\
T_2 &=& R\circ H_{a_2}(x) \\
T_3 &=& H_{a_3}\circ R^{-1} \\
T_4 &=& R\circ H_{a_4}\circ R^{-1},
\end{eqnarray*}
where $R:S^1\to S^1$ is the rotation $R(x)=x+1/4\mod{1}$;  see
Figure \ref{fig:4T}.
\begin{figure}[tb]
    \centering
    \psfrag{T1}{$T_1$}
    \psfrag{T2}{$T_2$}
    \psfrag{T3}{$T_3$}
    \psfrag{T4}{$T_4$}
        \includegraphics[width=0.90\textwidth]{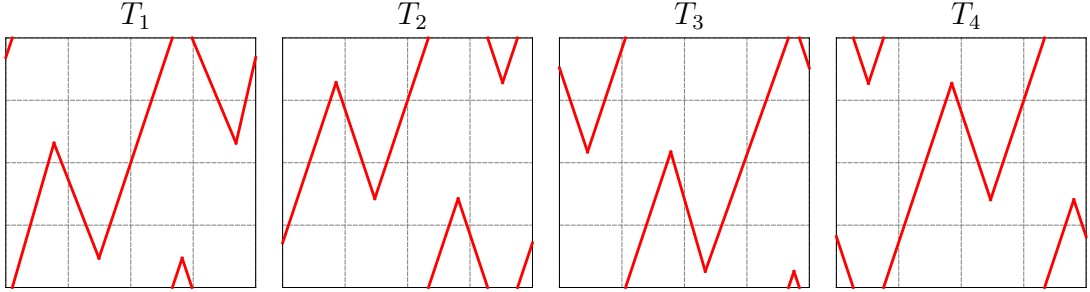}
    \caption{Graphs of $T_i$ for $i=1,\ldots,4$.}
    \label{fig:4T}
\end{figure}
\end{example}

%\paragraph{Cocycles of Perron--Frobenius Operators}
Let $m$ denote normalised Lebesgue measure on $M$. To each map
$T_{\omega}$ we associate a Perron--Frobenius operator
$\mathcal{P}_{{\omega}}:L^1(M,m)\circlearrowleft$ defined by
$\mathcal{P}_{{\omega}}f=\sum_{y\in T^{-1}_{\omega}x}{f(y)}/{|\det
DT_{\omega}(y)|}$. The operator $\mathcal{P}_\omega$ is a linear operator that acts on integrable functions in analogy to the action of $T_\omega$ on points.  If $f\in L^1(M,m)$ represents a density function for an ensemble of initial conditions, then $\mathcal{P}_\omega f$ represents the density function of the ensemble after the action of $T_\omega$ has been applied to the ensemble.  The map cocycle $\Phi$ naturally generates a
\emph{Perron--Frobenius cocycle}
$\mathcal{P}^{(k)}_\omega=\mathcal{P}_{{\sigma^{k-1}\omega}}\circ\cdots\circ
\mathcal{P}_{{\sigma\omega}}\circ \mathcal{P}_{{\omega}}$.
This composition of $k$ Perron--Frobenius operators capture the action on a function $f$ after $k$ iterations of the non-autonomous system.

%[GF:
%DITCH THIS EXAMPLE AS IT DOESN'T REALLY SAY ANYTHING?  REMOVE THIS
%SUBSUBSECTION HEADING?]
%\begin{example}
%\label{pfcocycleeg}
%We denote by $\mathcal{P}_\omega$ the Perron--Frobenius operator corresponding to $T_\omega\in\mathcal{T}$ acting on $(\mathrm{BV},\|\cdot\|_{BV})$, where $BV$ denotes the space of functions of bounded variation and $\|\cdot\|_{BV}=\max\{\var(\cdot),\|\cdot\|_{L^1}\}$.
%%Note that since the modulus of the gradient of each $T_i\in\mathcal{T}$ is equal to $3$, we have $\vartheta(\omega)=\log 1/3$ for any $\omega\in\Omega$. [GF: NEED DEFINITION AND NAME FOR $\vartheta$ and citation].
%\end{example}

\subsection{Continuous time -- Flows}
Let $F:\Omega\times M\to\mathbb{R}^d$ be a sufficiently regular
vector field. More precisely, we suppose that $F$ satisfies the
conditions of \cite[Theorem 2.2.2]{arnoldbook}, which will guarantee
the existence of a classical solution of the nonautonomous ODE
$\dot{x}(t)=F(\theta^t\omega,x(t))$, $t\in\mathbb{R}$.

To be concrete about the probability space  $(\Omega,\mathcal{F},\mathbb{P})$ in the continuous time setting, we may set $\Omega=\Xi\subset \mathbb{R}^{d_1}$, $d_1\ge 3$, and consider an autonomous ODE $\dot{z}=g(z)$ on $\Xi$.
Denote the flow for this ODE by $\xi:\mathbb{R}\times\Xi\to\Xi$ and suppose that $\xi$ preserves the probability measure $\mathbb{P}$;  that is, $\mathbb{P}\circ\xi(-t,\cdot)=\mathbb{P}$ for all $t\in\mathbb{R}$.
Thus, the autonomous, aperiodic flow $\xi$ drives the nonautonomous ODE
\begin{equation}
 \label{ode}
\dot{x}(t)=F(\theta^t\omega,x(t))=F(\xi(t,z),x(t)).
\end{equation}
We think of points $z\in \Xi$ as representing generalised time. We
assume that $(\Xi,\xi,\mathbb{P})$ is ergodic in the sense that if
$\xi(-t,\tilde{\Xi})=\tilde{\Xi}$ for some $\tilde{\Xi}\subset\Xi$
and for all $t\ge 0$ then $\mathbb{P}(\tilde{\Xi})=0$ or 1.

%\begin{remark}
%It is common to simply write $\dot{x}(t)=F(t,x(t))$, where
%$t\in\mathbb{R}$ is a scalar representing time. For the formal
%statements we shall therefore use generalised time $z\in\Xi$.  %In
%%practice, however, one necessarily computes over finite periods of
%%time and so when describing our numerical algorithms and numerical
%%experiments for continuous time we revert to scalar time.
%\end{remark}

Denote by $\phi:\mathbb{R}\times\Xi\times M\to M$ the flow for (\ref{ode}).
The flow $\phi$ satisfies $\frac{d}{dt}\phi(t,z,x)=F(\xi(t,z),\phi(t,z,x))$.

\begin{definition}
\label{invariantmeasurects} Endow $M$ with the Borel
$\sigma$-algebra and let $\mu$ be a probability measure on $M$.  We
call $\mu$ an \emph{invariant measure} if
$\mu\circ\phi(-t,z,\cdot)=\mu$ for all $z\in \Xi$ and $t\in
\mathbb{R}$.
\end{definition}

\begin{remark}
 \label{invmeasremark}
In Definition \ref{invariantmeasurects} we are insisting that $\mu$
is preserved at all ``time instants''.  As in the discrete time setting, more generally one may
allow $\mu=\mu_z$ and insist that
$\mu_{\xi(-t,z)}\circ\phi(-t,z,\cdot)=\mu_z$. However, as we will soon begin to focus on
coherent sets rather than invariant measures, we will restrict the
invariant measure to a ``time independent'' measure for clarity of
presentation. This is perfectly reasonable for one of the main
applications we have in mind, namely, aperiodically driven fluid
flow where $\mu\equiv$ Lebesgue, and volume is preserved by the flow
at all times.
\end{remark}

\begin{example}\label{ex:LZ}
%Consider the following time-dependent system on
%$M=[0,2\pi]\times[0,\pi]$, $t\in\mathbb{R}$:
%
%\begin{equation}\label{eq::LZwave}
%\begin{split}
%\dot{x} &=  c-A\sin(x-\nu t)\cos(y)\qquad\text{mod}2\pi\\
%\dot{y} &=  A\cos(x-\nu t)\sin(y)\\
%\end{split}
%\end{equation}
%where we set $c=0.5$, $A=1$, and the phase speed to $\nu=0.25$. The
%detailed description of this flow will be given in
%Section~\ref{ctscomputationssect}. We use the Lorenz flow as our
%driving system, $\xi$, and let its $x_1$-coordinate represent the
%generalized time for~\eqref{eq::LZwave}:
%\begin{equation}
%\begin{split}
%\dot{x_1}&=\sigma(x_2-x_1)\\
%\dot{x_2}&=\rho x_1-x_2-x_1x_3/\tau\\
%\dot{x_3}&=-\beta x_3+x_1x_2/\tau\\
%\end{split}
%\end{equation}
Consider the following nonautonomous system on a cylinder
$M=S^1\times [0,\pi]$. Let $\xi:\mathbb{R}\times \mathbb{R}^3\to
\mathbb{R}^3$ denote the flow for the driving system generated by
the Lorenz system of ODEs (\ref{eq::LZwave2a})--(\ref{eq::LZwave2c})
with standard parameters $\sigma=10$, $\beta=8/3$, $\rho=28$.
\begin{eqnarray}
\label{eq::LZwave2a}\dot{z_1}&=&\sigma(z_2-z_1)/\tau\\
\label{eq::LZwave2b}\dot{z_2}&=&(\rho z_1-z_2-z_1z_3)/\tau\\
\label{eq::LZwave2c}\dot{z_3}&=&(-\beta z_3+z_1z_2)/\tau.
\end{eqnarray}
It is well known that this Lorenz flow possesses an SBR measure
$\mathbb{P}$~\cite{tucker_99}. Let the time-dependent vector field
$F:\mathbb{R}\times S^1\times [0,\pi]\to S^1\times [0,\pi]$ generate
our non-autonomous ODE $(\dot{x}(t),\dot{y}(t))=F(\xi(t,z),x(t),y(t))$.
Explicitly,
%$(\dot{x}(t),\dot{y}(t))=F(\xi(t,z),x(t),y(t))$, where
\begin{eqnarray}
\label{eq::LZwave1a}\dot{x} &=& c-A\sin(x-\nu z_1(t))\cos(y)\qquad\mod{2\pi}\\
\label{eq::LZwave1b}\dot{y} &=& A\cos(x-\nu z_1(t))\sin(y),
\end{eqnarray}
with $c=0.5$, $A=1$, $\nu=0.25$. We set initial condition $z(0)=(0,1,1.5)$ and take the $z_1$-coordinate of the
Lorenz driving system to represent the generalized time for the
vector field $F(\xi(t,z),x(t),y(t))$. We use a scaling factor of $\tau=6.6685$
so that the temporal and spatial variation of $z_1(t)$ is similar to
that of the ``actual'' time $t$. Since $F(\xi(t,z),x,y)$ is
differentiable and bounded on $M$ for all $t$, classical solutions
of the nonautonomous ODE (\ref{eq::LZwave1a})--(\ref{eq::LZwave1b})
exist. The system~(\ref{eq::LZwave2a})--(\ref{eq::LZwave1b})
uniquely generates an RDS, see \cite[Theorem 2.2.2]{arnoldbook}.
In Figure~\ref{fig::LZwave} we demonstrate a trajectory of three different initial
points.
% starting inside the heteroclinic loops;  see
%Figure~\ref{fig::vectorfield_nodisturb}. %[GF: PERHAPS A BETTER
%FIGURE COULD BE TO HAVE A TRAJECTORY OF ONE POINT INSIDE AND ONE
%POINT OUTSIDE THE HETEROCLINIC LOOP...WE JUST WANT TO SHOW A COUPLE
%OF TYPICAL TRAJECTORIES IN THIS FIGURE]
\begin{figure}[tb]
\centerline{\includegraphics[scale=0.5]{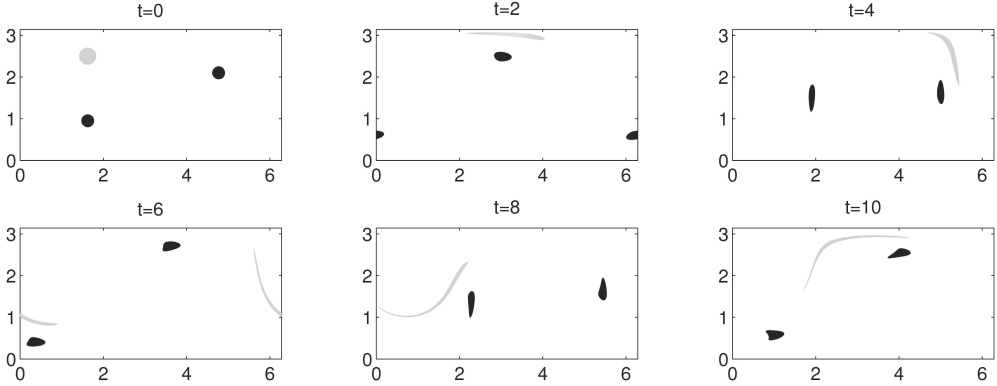}}
\caption{Trajectory of the time-dependent
system~\eqref{eq::LZwave1a}--\eqref{eq::LZwave1b} driven by the
Lorenz system at generalized times $\xi(t,z)$ }\label{fig::LZwave}
\end{figure}
\end{example}

%\subsubsection{Semigroups of Perron--Frobenius Operators}

We may define a family of Perron--Frobenius operators as
$\mathcal{P}_z^{(t)}f(x)=f(\phi(-t,\xi(t,z),x))\cdot|\det
D\phi(-t,\xi(t,z),x)|$ for $t\ge 0$. This family is a semigroup in $t$ as
$\mathcal{P}_z^{(t_1+t_2)}f=\mathcal{P}_{\xi(t_1,z)}^{(t_2)}\mathcal{P}_z^{(t_1)}f$.

\section{Galerkin projection and matrix cocycles}

Let $\mathcal{B}_n=\sp\{\chi_{B_i}:B_i\in \mathfrak{B}\}$ where
$\mathfrak{B}=\{B_1,\ldots,B_n\}$ is a partition of $M$ into connected
sets of positive Lebesgue measure. Define a projection
$\pi_n:L^1(M,m)\to\mathcal{B}_n$ as
\begin{equation}
 \label{pineqn}
\pi_nf=\sum_{i=1}^n \frac{\int_{B_i} f\,\mathrm{d}m}{m(B_i)}\chi_{B_i}.
\end{equation}
Following Ulam \cite{ulam}, in the sequel we will consider the finite rank operators
$\pi_n\mathcal{P}_\omega^{(1)}:L^1(M,m)\to\mathcal{B}_n$ and
$\pi_n\mathcal{P}_z^{(1)}:L^1(M,m)\to\mathcal{B}_n$, and the matrix
representations of the restrictions of
$\pi_n\mathcal{P}_\omega^{(1)}$ and  $\pi_n\mathcal{P}_z^{(1)}$ to
$\mathcal{B}_n$. We denote these matrix representations (under
multiplication on the right) by $P(\omega)$ and $P(z)$. Extending
Lemma 2.3 \cite{li} in a straightforward way to
the nonautonomous setting, one has
\begin{equation}
 \label{discreteulam}
P(\omega)_{ij}=\frac{m(B_j\cap \Phi(-1,\sigma\omega,B_i))}{m(B_i)}
\end{equation}
and %[SL: BOTH DEFINITIONS TRANSPOSED FOR LEFT MULTIPLICATION.]
\begin{equation}
 \label{continuousulam}
P(z)_{ij}=\frac{m(B_j\cap \phi(-1,\xi(1,z),B_i))}{m(B_i)}
\end{equation}
In particular, these matrices are numerically accessible.

\begin{remark}
Note we do not concern ourselves at all with the relationship
between $\mathcal{P}_{\omega}^{(1)}$ and
$\pi_n\mathcal{P}_{\omega}^{(1)}$;  this is a subtle
relationship and beyond the scope of this work. See
\cite{li,garysbr,ding,murraythesis,froylandrandomulam,bkl,murraynonuniform} for work
in this direction.
\end{remark}

The matrices $P(\omega)$ and $P(z)$ generate matrix cocycles
\begin{equation}
 \label{discretematrixcocycle}
P^{(k)}(\omega):=P(\sigma^{k-1}\omega)\cdots P(\sigma\omega)\cdot
P(\omega)
\end{equation}
and
\begin{equation}
 \label{ctsmatrixcocycle}
P^{(k)}(z):=P(\xi(k-1,z))\cdots P(\xi(1,z))\cdot P(z).
\end{equation}

\section{Discretised Oseledets functions and the Multiplicative
Ergodic Theorem} \label{METsect}
%We consider here the discrete time case only as this is the setting in which we are guaranteed existence of an Oseldets splitting, see Theorem \ref{thm:main} below.
%In the continuous time case, we may consider for example time-1 maps.
%[Note Arnold Theorem 3.4.11 produces a continuous time MET for
%invertible matrices by fairly briefly adding in a natural additional
%condition on $P$ to be satisfied for times $0\le t\le 1$. Perhaps a
%similar extension to Thm \ref{thm:main} is possible, though not for
%this paper].
In periodically driven flows, Liu and Haller \cite{liu_haller_04}
and Pikovsky and Popovych \cite{pikovsky_popovych_03}, observed that
certain tracer patterns persisted for long times before eventually
relaxing to the equilibrium tracer distribution.  Pikovsky and
Popovych \cite{pikovsky_popovych_03} recognised these patterns as
graphs of eigenfunctions of a Perron--Frobenius operator
corresponding to an eigenvalue $L<1$. These eigenfunctions decay
over time and the closer $L$ is to 1, the slower the decay and the
more slowly an initial tracer distribution will relax to
equilibrium. We now develop a framework for the considerably more difficult
aperiodic setting.

%Let us begin with a rough description of how Lyapunov exponents are
%connected with slow decay of functions.
%Oseledets subspaces are connected to strange eigenmodes in the
%discrete time setting.
Consider some suitable Banach space $(\mathcal{F},\|\cdot\|)$ of
real valued functions;  $\mathcal{F}$ is the function class in which
we search for slowly decaying functions. Suppose that the norm is
chosen so that for each $\omega\in\Omega$ and $k\ge 0$, the operator
$\mathcal{P}^{(k)}_\omega$ is Markov;
 that is, $\|\mathcal{P}^{(k)}_\omega\|=1$ for all $\omega$ and $k\ge
0$.  For $f\in\mathcal{F}$, we calculate the following limit:
\begin{equation}
\label{lambdaeqn}
\lambda(\omega,f)=\limsup_{k\to\infty}\frac{1}{k}\log\|\mathcal{P}^{(k)}_\omega f\|.
\end{equation}
We refer to $\lambda(\omega,f)\le 0$ as the \textit{Lyapunov
exponent} of $f$. If $f$ decays under the action of the
Perron--Frobenius operators at a geometric rate of $r^k$, $0<r<1$,
then $\lambda(\omega,f)=\log r$. The closer $r$ is to 1, the slower
the decay.  The extreme case of $r=1$ (no decay) is exhibited when
$f$ is the density of the invariant measure $\mu$ that is common to
all maps in our nonautonomous dynamical system.  We define the
\textit{Lyapunov spectrum}
$\Lambda(\mathcal{P},\omega):=\{\lambda(\omega,f):f\in\mathcal{F}
\}$.  In the aperiodic setting the new mathematical objects that are
analogous to strange eigenmodes and persistent patterns will be
called \emph{Oseledets functions.}
\begin{definition}
\label{strangeeigenmodedefn} \emph{Oseledets functions} correspond
to $f$ for which (i) $\lambda(\omega,f)$ is near zero and (ii) the
value $\lambda(\omega,f)$ is an isolated point in the
\textit{Lyapunov spectrum}.
\end{definition}
By considering $(\mathcal{F},\|\cdot\|)=(\mathcal{B}_n,\|\cdot\|_1)$, the actions of $\mathcal{P}_\omega^{(k)}$ and $\mathcal{P}^{(k)}_z$ are described by $P^{(k)}(\omega)$ and $P^{(k)}(z)$, respectively.
We may replace $\mathcal{P}_\omega^{(k)}$ and $\mathcal{P}^{(k)}_z$ in (\ref{lambdaeqn}) by $P^{(k)}(\omega)$ and $P^{(k)}(z)$, respectively, to obtain
%and $P^{(k)}(\omega)$ or $P^{(k)}(z)$ in place of
%$\mathcal{P}^{(k)}_\omega$
a standard setting where the possible values of $\lambda(\omega,f)$
are the Lyapunov exponents of cocycles of $n\times n$ matrices, and
\begin{equation}
\label{discretespectrum}
\Lambda(P,\omega):=\left\{\lim_{k\to\infty}\frac{1}{k}\log\|{P}^{(k)}(\omega)
f\|_1: f\in\mathcal{B}_n\right\}\end{equation}
 and
 \begin{equation}
\label{ctsspectrum}
\Lambda(P,z):=\left\{\lim_{k\to\infty}\frac{1}{k}\log\|{P}^{(k)}(z)
f\|_1: f\in\mathcal{B}_n\right\},\end{equation} exist for $\mathbb{P}$ almost-all $\omega\in\Omega$, and consist of at most
$n$ isolated points, $\lambda_n<\cdots<\lambda_1=0$. Of particular
interest to us is the function $f_2(\omega)$ (or $f_2(z)$)
 in $\mathcal{B}_n$, which represents the function that decays at the
slowest possible geometric rate $\lambda_2$.
\begin{remark}
In certain settings, this matrix cocycle \emph{exactly} captures all
large isolated Lyapunov exponents of the operator cocycle
$\mathcal{P}:(\mathrm{BV},\|\cdot\|_{\mathrm{BV}})\circlearrowleft$.
One such setting is a map cocycle formed by composition of piecewise
linear expanding maps with a common Markov partition
$\mathfrak{B}=\{B_1,\ldots,B_n\}$; %the matrix cocycle
%$P:(\mathcal{B}_n,\|\cdot\|_1)\circlearrowleft$ contains all
%Lyapunov exponents greater than $\vartheta(\omega)$ of the operator
%cocycle $\mathcal{P}$,
see \cite{froyland_lloyd_quas}.
\end{remark}

The following example illustrates the concept of Lyapunov spectrum
and Oseledets functions in the familiar autonomous setting. For the
remainder of this section, we adopt the discrete time notation of
$\sigma$ and $\omega$.

\begin{example}[``Autonomous'' single map]
\label{singlemapeg}  In \cite{foureggs} individual maps are
constructed for which the Perron--Frobenius operator has at least
one non-unit isolated eigenvalue when acting on the Banach space
$(BV,\|\cdot\|_{BV})$. A single autonomous map may be regarded as a
cocycle over a one-point space $\Omega=\{\omega\}$, and so we may
drop the dependence on $\omega$ in notation. Keller \cite{keller_84}
shows that for a piecewise expanding map $T$ of the interval $I$,
the spectrum of the associated Perron--Frobenius operator
$\mathcal{P}$ has an essential spectral radius $\rho_{\rm ess}(\mathcal{P})$
equal to the asymptotic local expansion rate $\sup_{x\in I} \lim_{k\to\infty}
\left|1/\mathrm{D}T^k(x)\right|^{1/k}$, and that there are at most
countably many spectral points, each isolated, of modulus greater
than $\rho_{\rm ess}(\mathcal{P})$.
%We refer to
%isolated Lyapunov exponents of $\Lambda(\mathcal{P},\omega)$ that
%lie in the interval $(\vartheta(\omega),1)$ as \emph{exceptional}.
In order to have an isolated spectral point, we construct a map of
$S^1$ which has an almost-invariant set (see Definition
\ref{aidefn}). The relation between almost-invariant sets and
isolated eigenvalues was noted in \cite{dellnitz_junge_99}. Consider
the partition $\mathfrak{B}=\{B_i: i=1,\ldots,6\}$, where
$B_i=((i-1)/6,i/6)$. Given $a\in\Z^6$, any map $T:S^1\to S^1$
defined by
\begin{eqnarray}\label{eq:Tformula}
T(x) = 3x-(i-1)/2+a_i/6\mod{1}, \quad x \in B_i
\end{eqnarray}
is Markov with respect to $\mathfrak{B}$.
Here we take $a=(0,0,1,4,3,3)$; see Figure \ref{fig:OneMap}. Notice that there is a low transfer of mass between the two intervals $[0,1/2]$ and $[1/2,1]$.
%Consider a map $T:S^1\to S^1$ given by
%\begin{eqnarray*}
%T(x) = \left\{\begin{array}{ll}
%3x                         & 0\leq x<\frac{1}{6},\\
%3x-\frac{1}{2} & \frac{1}{6} \leq x <\frac{1}{3},\\
%3x-\frac{5}{6} & \frac{1}{3} \leq x <\frac{2}{3},\\
%3x-\frac{3}{2} & \frac{2}{3} \leq x <\frac{5}{6},\\
%3x-2                       & \frac{5}{6} \leq x \leq 1.
%\end{array}\right.
%\end{eqnarray*}
Since $\mathfrak{B}$ is a Markov partition for $T$, the space of characteristic functions $\mathcal{B}_6 =\{\chi_{B_i}:i=1,\ldots,6\}$ is an invariant subspace of $\mathrm{BV}$ for the Perron--Frobenius operator $\mathcal{P}$ of $T$. Thus the action of $\mathcal{P}_\omega = \mathcal{P}$ on $\mathcal{F}=\mathcal{B}_6$
is represented by the matrix
\begin{eqnarray}P = P(\omega) = \frac{1}{3}
\left(\begin{array}{cccccc}
1 & 1 & 0 & 1 & 0 & 0 \\
1 & 1 & 1 & 0 & 0 & 0 \\
1 & 1 & 1 & 0 & 0 & 0 \\
0 & 0 & 1 & 0 & 1 & 1 \\
0 & 0 & 0 & 1 & 1 & 1 \\
0 & 0 & 0 & 1 & 1 & 1 \\
\end{array}\right),
\end{eqnarray}
which has non-zero eigenvalues $L_1=1,L_2=(1+\sqrt{2})/3,
L_3=(1-\sqrt{2})/3$. The map $T$ is piecewise affine with constant
slope $3$ and so the logarithm of the local expansion rate is $\log (1/3)$.

The eigenvalue $L_2\approx 0.805$ of $P$ thus gives rise to an
isolated point $\lambda_2\approx \log 0.805$ in the Lyapunov
spectrum $\Lambda(\mathcal{P})$. The corresponding Oseledets
function $f_2$ is given by
$f_2(x)=\sum_{i=1}^6w_{2,i}\chi_{B_i}(x)$, where $w_2$ is the
eigenvector of $P$ corresponding to the eigenvalue $L_2$, see Figure
\ref{fig:OneMap}.
%
%, whose positive an negative supports identify the intervals $[0,1/2]$ and $[1/2,1]$. The slow transfer of mass between these two intervals can be seen from the graph.
\begin{figure}[tb]
    \centering
    \psfrag{T}{$T$}
    \psfrag{v}{$f_2$}
        \includegraphics[width=0.80\textwidth]{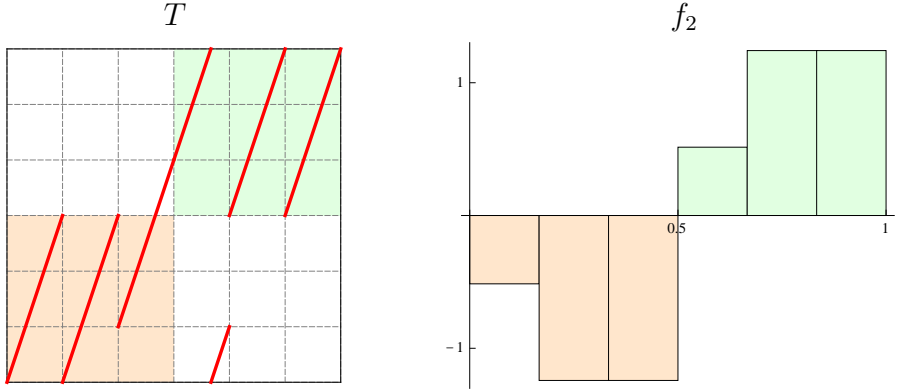}
    \caption{Graph of $T$ and Oseledets function $f_2$.}
    \label{fig:OneMap}
\end{figure}
Since $|L_3|\approx 0.138<1/3$, this means that $\log L_2$ is the
unique isolated Lyapunov exponent in $\Lambda(\mathcal{P})$. Note
that the set $\{f_2>0\}$ corresponds to the set $[0,1/2]$. We will
discuss this property further in Section \ref{cssection}.
\end{example}

%\vspace{2mm}
%\noindent[{\footnotesize FOR SIMON: Insert here eg. 3 simple piecewise
%linear Markov maps so that the ``triple-product'' PF operator has an
%isolated eigenvalue. We then say $\Omega$ consists of three points,
%and the strange eigenmodes are defined via the eigenvectors of
%$P^{(3)}(\omega)$. Reference appropriate papers. Perhaps recycle
%example from \cite{froyland_lloyd_quas}?  See also
%\cite{froyland_padberg} for a detailed example of similar
%calculations for a periodically driven flow.}]

\begin{example}[Periodic map cocycle]
\label{periodiceg} We construct a periodic map cocycle from a
collection of maps with a common Markov partition. The map cocycle
is formed by cyclically composing three maps of $S^1$. Consider the
sequence space $\Omega=\{\omega\in\{1,2,3\}^\mathbb{Z}:\forall
i\in\mathbb{Z}, M_{\omega_i,\omega_{i+1}}=1\}$ where
$M=\left(\begin{array}{cccc}0&1&0\\ 0&0&1\\ 1&0&0\\
\end{array}\right)$. We consider $\mathcal{T}=\{T_j:j=1,2,3\}$,
where $T_j$ is given by (\ref{eq:Tformula}) with parameter
$a^{(j)}$, where
$$
a^{(1)}=(3,2,2,0,5,5),\quad a^{(2)}=(2,1,4,5,4,1),\quad a^{(3)}=(1,3,3,4,0,0),
$$
see Figure \ref{fig:PerMaps}.
\begin{figure}[tb]
    \centering
    \psfrag{T1}{$T_1$}
    \psfrag{T2}{$T_2$}
    \psfrag{T3}{$T_3$}
        \includegraphics[width=0.90\textwidth]{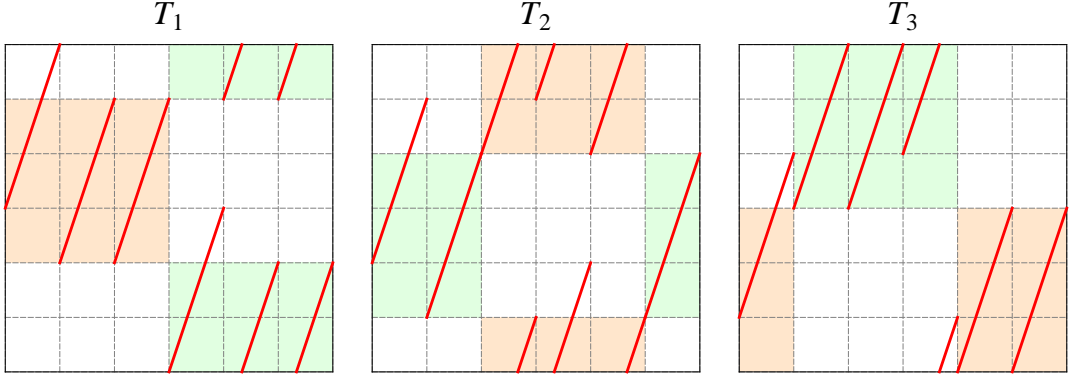}
    \caption{Graphs of $T_1$, $T_2$ and $T_3$. }
    \label{fig:PerMaps}
\end{figure}
As in Example \ref{singlemapeg} we look for Lyapunov exponents that
are strictly greater than the logarithm of the asymptotic \emph{local expansion rate}
\begin{eqnarray}
\sup_{x\in I}\lim_{k\to\infty} | 1/D(T_3\circ T_2\circ
T_1)^k|^{1/3k}.
%\vartheta(\omega):=\inf_{x\in M} \limsup_{n\to\infty}
%\log |\mathrm{D}\Phi^{(n)}(x)(\omega)|.
\end{eqnarray}
As each map $T_j$ is piecewise affine with constant slope $3$, the logarithm of the local expansion rate is $\log(1/3)$. Note also that $T_1$ approximately maps $[0,1/2]$
to $[1/3,5/6]$, $T_2$ then maps $[1/3,5/6]$ approximately to
$[0,1/6]\cup [2/3,1]$, and finally $T_3$ maps $[2/3,1/3]$
approximately back to $[0,1/2]$. Each map $T_j$ leaves the space
$\mathcal{B}_6$ from Example \ref{singlemapeg} invariant, and thus
the Perron--Frobenius operator $\mathcal{P}_j$ of $T_j$ restricted
to $\mathcal{B}_6$ has matrix representation $P_j$, where $3P_j$,
$j=1,2,3$, are respectively
\begin{eqnarray}
\left(\begin{array}{cccccc}
0 & 0 & 0 & 1 & 1 & 1 \\
0 & 0 & 0 & 1 & 1 & 1 \\
0 & 1 & 1 & 1 & 0 & 0 \\
1 & 1 & 1 & 0 & 0 & 0 \\
1 & 1 & 1 & 0 & 0 & 0 \\
1 & 0 & 0 & 0 & 1 & 1 \\
\end{array}\right),
%P_2 = \frac{1}{3}
\left(\begin{array}{cccccc}
0 & 0 & 1 & 1 & 1 & 0 \\
0 & 1 & 0 & 1 & 0 & 1 \\
1 & 1 & 0 & 0 & 0 & 1 \\
1 & 1 & 0 & 0 & 0 & 1 \\
1 & 0 & 1 & 0 & 1 & 0 \\
0 & 0 & 1 & 1 & 1 & 0 \\
\end{array}\right),
%P_3 = \frac{1}{3}
\left(\begin{array}{cccccc}
0 & 0 & 0 & 1 & 1 & 1 \\
1 & 0 & 0 & 0 & 1 & 1 \\
1 & 0 & 0 & 0 & 1 & 1 \\
1 & 1 & 1 & 0 & 0 & 0 \\
0 & 1 & 1 & 1 & 0 & 0 \\
0 & 1 & 1 & 1 & 0 & 0 \\
\end{array}\right).
\end{eqnarray}
The triple product $P^{(3)}(\omega)=P_3P_2P_1$ has non-zero
eigenvalues $L_1=1, L_2=(13+ \sqrt{233})/54$ and
$L_3=(13-\sqrt{233})/54$. Since $L_2\approx 0.523$, its associated
eigenvector $w_2$ satisfies $\lambda(\omega,w_2)= \log
\sqrt[3]{\lambda_2}>\log (1/3)$. Since $P^{(3)}(\sigma^k\omega)$,
$k=1,2$, are cyclic permutations of the factors of
$P^{(3)}(\omega)$, they share the same eigenvalues, and in
particular $L_2$. Thus $(1/3)\log L_2$ is an isolated Lyapunov
exponent of $\Lambda(\mathcal{P},\omega)$ for each
$\omega\in\Omega$. Associated to the eigenvalue $L_2$, the matrices
$P(\sigma^k\omega)$, $k=0,1,2$, have corresponding eigenvectors
$w_2(\sigma^k\omega)$. The three vectors $w_2(\sigma^k\omega)$,
$k=0,1,2$, generate the periodic Oseledets functions
$f_2(\sigma^k\omega)=\sum_{i=1}^6
w_{2,i}(\sigma^{k\mod{3}}\omega)\chi_{B_i}$, see Figure \ref{fig:PerCS}.
\begin{figure}[h]
    \centering
    \psfrag{v11}{$f_2(\omega)$}
    \psfrag{v12}{$f_2(\sigma\omega)$}
    \psfrag{v13}{$f_2(\sigma^2\omega)$}
        \includegraphics[width=0.90\textwidth]{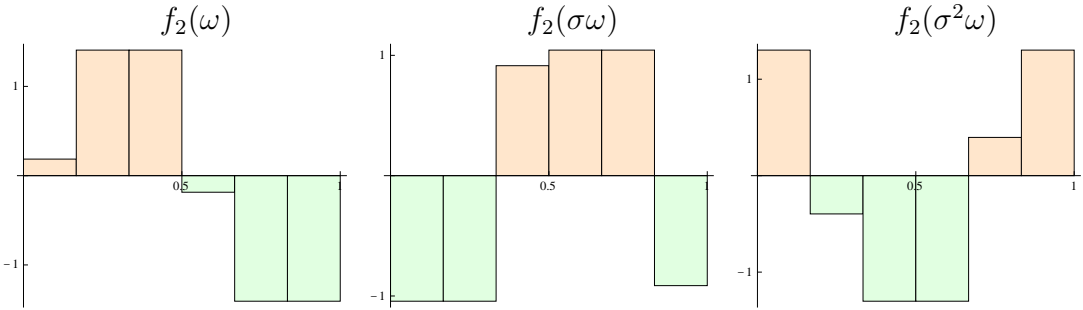}
    \caption{Oseledets functions $f_2(\sigma^k\omega)$ for $k=0,1,2$.}
    \label{fig:PerCS}
\end{figure}
Note that the sets $\{f_2(\sigma^k\omega)>0\}_{k=0,1,2}$ correspond
to the sets $[0,1/2]$, $[1/3,5/6]$, and $[0,1/6]\cup [2/3,1]$,
respectively. We
will discuss this property further in Section \ref{cssection}.
% In FLQ08 it is shown that for map cocycle with a common Markov partition, all exceptional Lyapunov exponents of the Perron--Frobenius cocycle $\mathcal{P}(\omega)$ acting on $\mathrm{BV}$ arise as Lyapunov exponents of the restricted cocycle $P(\omega)$. Thus $(\log\lambda_2)/3$ is the largest exceptional Lyapunov exponent of $\Lambda(\mathcal{P}(\omega))$.
%[GF: I REMOVED FINAL PARA FROM THIS EG. AS IT SEEMS TO REPEAT EARLIER STATEMENTS]
For another such example see \cite{froyland_lloyd_quas}. See also
\cite{froyland_padberg_08} for a detailed example of similar
calculations for a periodically driven flow.
\end{example}

% \begin{example}[Autonomous single map]
% \label{singlemapeg} FOR SIMON: Insert here a simple piecewise linear
% Markov map so that the PF operator has an isolated eigenvalue. We
% then say $\Omega$ is a single point, $P_\omega$ is the matrix
% (matrix here), and the strange eigenmodes are defined via the
% eigenvectors. Reference appropriate papers.
% \end{example}
% \begin{example}[Periodic map cocycle]
% \label{periodiceg} FOR SIMON: Insert here eg. 3 simple piecewise
% linear Markov maps so that the ``triple-product'' PF operator has an
% isolated eigenvalue. We then say $\Omega$ consists of three points,
% and the strange eigenmodes are defined via the eigenvectors of
% $P^{(3)}(\omega)$. Reference appropriate papers. Perhaps recycle
% example from \cite{froyland_lloyd_quas}?  See also
% \cite{froyland_padberg_08} for a detailed example of similar
% calculations for a periodically driven flow.
% \end{example}

In the nonautonomous setting, we can no longer easily construct
Oseledets functions as eigenfunctions of a single operator, or
eigenvectors of a single matrix. In fact, the Oseledets functions
are themselves (aperiodically) time dependent in the nonautonomous
setting. Our model of Oseledets functions for nonautonomous systems
is, as the name suggests, built around the Multiplicative Ergodic
Theorem, see e.g.\ \cite[Chapter 3, \S4]{arnoldbook}. We now state a
strengthened version \cite{froyland_lloyd_quas} of the
Multiplicative Ergodic Theorem that we require for our current
purposes.
\begin{theorem}[\cite{froyland_lloyd_quas}]
\label{thm:main} Let $\sigma$ be an invertible ergodic
measure-preserving transformation of the space
$(\Omega,\mathcal{H},\mathbb{P})$. Let $P:\Omega\to M_n(\R)$ be a
measurable family of matrices satisfying
$$\int \log^+\|P(\omega)\|\,\mathrm{d}\mathbb{P}(\omega)<\infty.
$$
Then there exist $\lambda_1>\lambda_2>\cdots>\lambda_\ell\geq -\infty$
and %%here
dimensions $m_1,\ldots,m_\ell$, with $m_1+\cdots+m_\ell=n$, and a
measurable family of subspaces $W_{j}(\omega)\subseteq \R^n$ such
that for almost every $\omega\in\Omega$ the following hold:
\begin{enumerate}
\item $\dim W_{j}(\omega)=m_j$;
  \item $\R^n=\bigoplus_{j=1}^\ell W_{j}(\omega)$;
  \item $P(\omega)W_{j}(\omega)\subseteq W_{j}(\sigma\omega)$ (with
    equality if $\lambda_j>-\infty$);
  \item for all $v\in W_{j}(\omega)\setminus\{0\}$, one has
    \begin{equation*}
      \lim_{k\to\infty}
      (1/k)\log\|P(\sigma^{k-1}\omega)\cdots P(\sigma\omega)\cdot P(\omega) v\|= \lambda_j.
    \end{equation*}
  \end{enumerate}
\end{theorem}
The subspaces $W_j(\omega)$ are the general time-dependent analogues
of the vectors $w_2$ and $w_2(\sigma^k\omega), k=0,1,2$ of Examples
\ref{singlemapeg} and \ref{periodiceg}, respectively. We may
explicitly construct a slowest decaying discrete Oseledets
function as $f_2(\omega):=\sum_{i=1}^n w_{2,i}(\omega)\chi_{B_i}$, where
$w_2(\omega)\in W_2(\omega)$. In the sequel, for brevity we will
often call $W_j(\omega)$ a subspace or a function, recognising its
dual roles.

\begin{remark}
\label{remark4a} We remark that if $\ell\ge 2$, $m_2=1$, and
$\lambda_2>-\infty$, the family of vectors
$\{f_2(\sigma^k\omega)\}_{k\ge 0}$ is the \emph{unique}\footnote{Assume
there is another family $\{w'_2(\sigma^k\omega)\}_{k\ge 0}\neq
\{w_2(\sigma^k\omega)\}_{k\ge 0}$ (up to scalar multiples) with
these properties.  Then $w'_2(\sigma^k\omega)=\sum_{j=2}^\ell
\alpha_{k,j} w_\ell(\sigma^k\omega)$ for some $\alpha_{k,j}$, $j=2,\ldots,\ell$,  with
$\alpha_{k,2}\neq 0$.  WLOG assume $\alpha_{k,2}, \alpha_{k,j'}\neq
0$ for some $2<j'\le \ell$ and all $k\ge 0$, but that $\alpha_{k,j}=0$
for all $j\neq 2,j'$ and all $k\ge 0$.  Then $m_2\ge 2$ in Theorem
\ref{thm:main}, a contradiction.} (up to scalar multiples) family of
vectors in $\mathcal{B}_n$ with the properties that
\begin{enumerate}
\item $\lim_{k'\to\infty}
      (1/k')\log\|P^{(k')}(\omega)f_2(\sigma^k\omega)\|_1=\lambda_2$, $k\ge 0$,
      \item $\pi_n\mathcal{P}_\omega
      f_2(\sigma^k\omega)=\alpha_k
      f_2(\sigma^{k+1}\omega)$ for some $\alpha_k\neq 0$, $k\ge 0$.
\end{enumerate}
\end{remark}

\begin{remark}
Theorem \ref{thm:main} strengthens the standard version of the MET
for one-sided time with noninvertible matrices (see e.g.\  \cite[Theorem
3.4.1]{arnoldbook}) to obtain the conclusions of the two-sided
time MET with invertible matrices (see e.g.\
\cite[Theorem 3.4.11]{arnoldbook}). In \cite[Theorem 3.4.1]{arnoldbook}, the
existence of only a \textit{flag}
$\mathbb{R}^n=V_1(\omega)\supset\cdots\supset V_\ell(\omega)$ of
Oseledets subspaces is guaranteed, while in
\cite[Theorem 3.4.11]{arnoldbook}, the existence of a \textit{splitting}
$W_1(\omega)\oplus\cdots\oplus W_\ell(\omega)=\mathbb{R}^n$ is
guaranteed. Theorem \ref{thm:main} above demonstrates existence of
an Oseledets \textit{splitting} for two-sided time with
\textit{noninvertible} matrices. This is particularly important for
our intended application as the projected Perron--Frobenius operator
matrices are non-invertible.
Recent further extensions \cite{froyland_lloyd_quas2} prove existence and uniqueness of Oseledets subpsaces for cocycles of Lasota-Yorke maps.
\end{remark}

\section{Numerical approximation of Oseledets functions}

In the autonomous and periodic settings we have seen in Examples
\ref{singlemapeg} and \ref{periodiceg} that the subspaces
$W_2(\omega)=\sp\{w_2(\omega)\}$ were one-dimensional, and that the
vectors $w_2(\omega)$ could be simply determined as eigenvectors of
matrices.
%
%to repeated multiplication by a fixed matrix $P$. The Oseledets
%splitting is then given simply by the eigenspaces of the matrix $P$.
%In the periodically driven setting with period $K$, $\Omega$
%consists of $K$ elements and $\theta$ cyclically permutes these
%elements. The matrix cocycle reduces to a periodic multiplication of
%$K$ matrices. Setting $P^{(K)}$ to be the product of these $K$
%matrices, the Oseledets splitting is then given simply by the
%eigenspaces of the matrix $P^{(K)}$.
For truly nonautonomous systems (those that are aperiodically
driven), the Oseledets splittings are difficult to compute. In this
section we outline a numerical algorithm to approximate the
$W_j(\omega)$ subspaces from Theorem \ref{thm:main}. The algorithm
is based on the push-forward limit argument developed in the proof
of Theorem \ref{thm:main}. To streamline notation, we describe the
discrete time and continuous time setting separately.

\subsection{Discrete time}

We first describe a simple and efficient method to construct the
matrix $P(\omega)$ defined in (\ref{discreteulam}).

\begin{algorithm}[Approximation of $P(\sigma^{-k}\omega)_{ij}$, $0\le k\le N$]
 \quad
\begin{enumerate}
 \item Partition the state space $M$ into a collection of connected sets $\{B_1,\ldots,B_n\}$ of small diameter.
\item Fix $i$, $j$, and $k$ and create a set of $Q$ test points $x_{j,1},\ldots,x_{j,Q}\in B_j$ that are uniformly distributed over $B_j$.
\item For each $q=1,\ldots,Q$ calculate $y_{j,q}=T_{\sigma^{-k}\omega}x_{j,q}$.
\item Set
\begin{equation}
 \label{discreteapproxulam}
P(\sigma^{-k}\omega)_{ij}=\frac{\#\{q:y_{j,q}\in B_i\}}{Q}
\end{equation}
\end{enumerate}
\end{algorithm}
We now describe how to use the matrices $P(\omega)$ to approximate
the subspaces $W_j(\omega)$. An intuitive description of the ideas behind Algorithm \ref{discreteeigenmodealg} immediately follows the algorithm statement.

\begin{algorithm}[Approximation of Oseledets subspaces
  $W_j(\omega)$ at $\omega\in\Omega$.]
  \label{discreteeigenmodealg}
\quad
 \begin{enumerate}
\item Construct the Ulam matrices $P^{(M)}(\sigma^{-N}\omega)$ and $P^{(N)}(\sigma^{-N}\omega)$ from (\ref{discreteapproxulam}) and (\ref{discretematrixcocycle}) for suitable $M$ and $N$.  The number $M$ represents the number of iterates over which one measures the decay, while the number $N$ represents how many iterates the resulting ``initial vectors'' are pushed forward to better approximate
elements of the $W_j(\omega)$.
\item Form $$
\Psi^{(M)}(\sigma^{-N}\omega):=(P^{(M)}(\sigma^{-N}\omega)^\top
P^{(M)}(\sigma^{-N}\omega))^{1/2M}
$$ as an approximation to the standard limiting matrix
$$
B(\sigma^{-N}\omega):=\lim_{M\to\infty}
\left(P^{(M)}(\sigma^{-N}\omega)^\top
  P^{(M)}(\sigma^{-N}\omega)\right)^{1/2M}
$$
appearing in the Multiplicative Ergodic Theorem (see e.g.\  \cite[Theorem
3.4.1(i)]{arnoldbook}).
\item Calculate the orthonormal eigenspace decomposition of
$\Psi^{(M)}(\sigma^{-N}\omega)$,
 denoted by $U^{(M)}_j(\sigma^{-N}\omega)$, $j=1,\ldots,\ell$.
We are particularly interested in low values of $j$, corresponding
to large eigenvalues $L_j$.
\item Define
  $W_j^{(M,N)}(\omega):=P^{(N)}(\sigma^{-N}\omega)U^{(M)}_j(\sigma^{-N}\omega)$
  via the push forward under the matrix cocycle.
\item $W_j^{(M,N)}(\omega)$ is our numerical approximation to $W_j(\omega)$.
\end{enumerate}
\end{algorithm}
Here is the idea behind the above algorithm. If we choose $M$ large
enough, the eigenspace $U^{(M)}_j(\sigma^{-N}\omega)$ should be
close to the limiting ($M\to\infty$) eigenspace
$U_j(\sigma^{-N}\omega)$. Vectors in the eigenspace
$U_j(\sigma^{-N}\omega)$ experience stretching at a rate close to
$L_j$.  Note that the eigenspace $U^{(M)}_j(\sigma^{-N}\omega)$ is the $j^{\rm th}$ singular vector of the matrix $P^{(M)}(\sigma^{-N}\omega)$, which experiences a ``per unit time'' average stretching from time $-N$ to $-N+M$ of $L_j$.
Choose some arbitrary $v\in U_j(\sigma^{-N}\omega)$ and
write $v=\sum_{j'=j}^\ell w_{j'}$ with $w_{j'}\in
W_{j'}(\sigma^{-N}\omega)$. Pushing forward by
$P^{(N)}(\sigma^{-N}\omega)$ for large enough $N$ will result in
$\|P^{(N)}(\sigma^{-N}\omega)w_j\|$ dominating
$\|P^{(N)}(\sigma^{-N}\omega)w_{j'}\|$ for $j<j'\le \ell$. Thus, for
large $M$ and $N$ we expect $W_j^{(M,N)}(\omega)$ to be close to
$W_j(\omega)$.

\begin{remarks}
\quad
\begin{enumerate}
\item Theorem \ref{thm:main} states that $W_j^{(\infty,N)}(\omega)\to W_j(\omega)$ as
$N\to\infty$.
\item This method may also be used to calculate the Oseledets
  subspaces for two-sided linear cocycles, and may be more convenient, especially for large $n$, than
  the standard method of intersecting the relevant subspaces of flags
  of the forward and backward cocycles.
 \end{enumerate}
\end{remarks}

The numerical approximation of the Oseledets subspaces has been
considered by a variety of authors in the context of (usually
invertible) nonlinear differentiable dynamical systems, where the
linear cocycle is generated by Jacobian matrices concatenated along
trajectories of the nonlinear system. Froyland \emph{et al.}
\cite{Froyland_Judd_Mees_95} approximate the Oseledets subspaces in
invertible two-dimensional systems by multiplying a randomly chosen
vector by $P^{(N)}(\sigma^{-N}\omega)$ (pushing forward) or
$P^{(-N)}(\sigma^{N}\omega)$ (pulling back, where $P^{(-N)}(\sigma^{N}\omega)=P^{-1}(\omega)\cdots P^{-1}(\sigma^{N-1}\omega)$. Trevisan and Pancotti
\cite{trevisan_pancotti_98} calculate eigenvectors of
$\Psi^{(M)}(\omega)$ for the three-dimensional Lorenz flow,
increasing $M$ until numerical convergence of the eigenvectors is
observed. Ershov and Potapov \cite{ershov_potapov_98} use an
approach similar to ours, combining eigenvectors of a $\Psi^{(M)}$
with pushing forward under $P^{(N)}$. Ginelli \emph{et al.}
\cite{ginelli_etal_07} embed the approach of
\cite{Froyland_Judd_Mees_95} in a QR-decomposition methodology to
estimate the Oseledets vectors in higher dimensions. In the
numerical experiments that follow, we have found our approach to
work very well, with fast convergence in terms of both $M$ and $N$.

\subsection{Continuous time}

As our practical computations are necessarily over finite time intervals, from now on, when dealing with continuous time systems, we will compute $P^{(k)}(z)$ as $\pi_n\mathcal{P}_z^{(k)}$ rather than as $P(\xi(k-1,z))\cdots P(\xi(1,z))\cdot P(z)$.
If the computation of $\pi_n\mathcal{P}_z^{(k)}$ can be done
accurately (this will be discussed further in Section
\ref{ctscomputationssect}), then this representation should be
closer to $\mathcal{P}_z^{(k)}$ as there are fewer applications of
$\pi_n$.

%\item Replace the ``generalised time'' $z\in\Xi$ with scalar time $t\in\mathbb{R}$.  That is, $P^{(\tau)}(t)=\pi_n\mathcal{P}_t^{(\tau)}$ will denote the Perron--Frobenius operator describing the nonautonomous flow from time $t$ for a duration $\tau$.  This simplification will streamline the notation in the algorithms, proofs, and numerical examples to follow.
%\end{enumerate}

We first describe a simple and efficient method to construct the matrix $P(\omega)$ defined in (\ref{discreteulam}).
\begin{algorithm}[Approximation of $P^{(M)}(\xi(-N,z))$, $N\ge 0$]
 \label{ctsPalg}
 \quad
\begin{enumerate}
 \item Partition the state space $M$ into a collection of connected sets $\{B_1,\ldots,B_n\}$ of small diameter.
\item Fix $i$,$j$, and $z$ and create a set of $Q$ test points $x_{j,1},\ldots,x_{j,Q}\in B_j$ that are uniformly distributed over $B_j$.
\item For each $q=1,\ldots,Q$ calculate $y_{j,q}=\phi(M,\xi(-N,z),x_{j,q})$.
\item Set
\begin{equation}
 \label{continuousapproxulam}
P^{(M)}(\xi(-N,z))_{ij}=\frac{\#\{q:y_{j,q}\in B_i\}}{Q}
\end{equation}
\end{enumerate}
\end{algorithm}

The flow time $M$ should be chosen long enough so that most test
points leave their partition set of origin, otherwise at the
resolution given by the partition $\{B_1,\ldots,B_n\}$, the matrix
$P^{(M)}(\xi(-N,z))$ matrix will be too close to the $n\times n$
identity matrix.  If the action of $\phi$ separates nearby points,
as is the case for chaotic systems, clearly the longer the flow
duration $M$, the greater $Q$ should be in order to maintain a good
representation of the images $\phi(M,\xi(-N,z),B_i)$ by the test points.

\begin{algorithm}[Approximation of Oseledets subspaces
  $W_j(z)$ at $z\in\Xi$.]
  \label{ctseigenmodealg}
\quad
\begin{enumerate}
\item Construct the Ulam matrices $P^{(M)}(\xi(-N,z))$ and $P^{(N)}(\xi(-N,z))$ from (\ref{continuousapproxulam}) for suitable $M$ and $N$.  The number $M$ represents the flow duration over which rate of decay is measured, while the number $N$ represents the duration over which the resulting ``initial vectors'' are pushed forward to better approximate elements of the $W_j(z)$.
\item Form $$
\Psi^{(M)}({\xi(-N,z)}):=\left(P^{(M)}(\xi(-N,z))^\top
P^{(M)}(\xi(-N,z))\right)^{1/2M}
$$ as an approximation to the standard limiting matrix
$$
B(\xi(-N,z)):=\lim_{M\to\infty} \left(P^{(M)}(\xi(-N,z))^\top
P^{(M)}(\xi(-N,z))\right)^{1/2M}$$ appearing in the Multiplicative Ergodic
Theorem (see e.g.\ \cite[Theorem 3.4.1(i)]{arnoldbook}).
\item Calculate the orthonormal eigenspace decomposition of
$\Psi^{(M)}(\xi(-N,z))$,
 denoted by $U^{(M)}_{j}(\xi(-N,z))$, $j=1,\ldots,\ell$. We are particularly interested in low values of $j$, corresponding to large eigenvalues $L_j$.
\item Define
  $W_{j}^{(M,N)}(z):=P^{(N)}(\xi(-N,z))U^{(M)}_{j}(\xi(-N,z))$
  via the push forward under the matrix cocycle.
\item $W_{j}^{(M,N)}(z)$ is our numerical approximation to $W_{j}(z)$.
\end{enumerate}
\end{algorithm}

\subsection{Continuity of the Oseledets subspaces in continuous
time} When treating continuous time systems, one may ask about the
continuity properties of $W_j^{(M,N)}(z)$ in $z$. In the following
we suppose that $W^{(M,N)}_2(z)$ is one-dimensional. For large $M$
and $N$, $W^{(M,N)}_2(z)$ will approximate the most dominant Oseledets subspace at time $z$. Suppose that we are interested in how this
subspace changes from time $z$ to time $\xi(\delta,z)$ for small
$\delta>0$. There are two ways to obtain information at time
$\xi(\delta,z)$. Firstly, we can simply push forward $W^{(M,N)}_2(z)$
slightly longer to produce $W^{(M,N+\delta)}_2(\xi(\delta,z))$. Secondly, we can
compute $\Psi^{(M)}$ slightly later at time $\xi(\delta,z)$ to produce
$W^{(M,N)}_2(\xi(\delta,z))$.

To compare the closeness of $W^{(M,N)}_2(z)$ to
$W^{(M,N+\delta)}_2(\xi(\delta,z))$ and $W^{(M,N)}_2(\xi(\delta,z))$, we represent
each as a function and make a comparison in the $L^1$ norm. We
assume that $U^{(M)}_2(\xi(-N,z))$ is one-dimensional
%with eigenvalue $\lambda_2(\xi(-N,z))$
and define
$f_{n,\xi(-N,z),M}=\sum_{i=1}^n (u^{(M)}_2(\xi(-N,z)))_i\chi_{B_i}\in L^1(M,m)$
where $u^{(M)}_2(\xi(-N,z))\in U^{(M)}_2(\xi(-N,z))$ is scaled so that
$\|f_{n,\xi(-N,z),M}\|_1=1$. Let
$\hat{f}_{n,z,M,N}=\pi_n\mathcal{P}^{(N)}_{\xi(-N,z)}f_{n,\xi(-N,z),M}$. Note
that $\hat{f}_{n,z,M,N}=\sum_{i=1}^n (w^{(M,N)}_2(z))_i\chi_{B_i}$
for some $w^{(M,N)}_2(z)\in W^{(M,N)}_2(z)$.

We firstly compare $\hat{f}_{n,\xi(\delta,z),M,N+\delta}$ and $\hat{f}_{n,z,M,N}$.
\begin{proposition}
$\|\hat{f}_{n,\xi(\delta,z),M,N+\delta}-\hat{f}_{n,z,M,N}\|_1\to 0$ as $\delta\to 0$.
\end{proposition}
\begin{proof}
Note that
$\hat{f}_{n,\xi(\delta,z),M,N+\delta}=\pi_n\mathcal{P}^{(N+\delta)}_{\xi(-N,z)}f_{n,\xi(-N,z),M}$
while $\hat{f}_{n,z,N,M}=\pi_n\mathcal{P}^{(N)}_{\xi(-N,z)}f_{n,\xi(-N,z),M}$.
The proof will follow from the result that
$\mathcal{P}_{z}^{(\tau)}$ is a \textit{continuous} semigroup;  that
is, $\lim_{\delta\to 0}\|\mathcal{P}_t^{(\delta)}f-f\|_1=0$ for all
$t\in \mathbb{R}$, $f\in L^1(M,m)$.
\begin{lemma}
\label{basicctylemma} $\|\mathcal{P}_{z}^{(\delta)}f-f\|_1\to 0$ as
$\delta\to 0$ for all $z\in \Xi$ and $f\in L^1$.
\end{lemma}
\begin{proof}
The proof runs as a non-autonomous version of the discussion in
Remark 7.6.2 \cite{lasota_mackey1}. Note that
$\mathcal{P}_{z}^{(\delta)}f(x)=f(\phi(-\delta,\xi(\delta,z),x))\cdot\det
D\phi(-\delta,\xi(\delta,z),x)$, where $\phi(-\delta,\xi(\delta,z),\cdot)$
denotes the flow from $\xi(\delta,z)$ in reverse time for duration
$\delta$. For the moment consider continuous $f$. Since $x\mapsto\phi(s,z,x)$ is at
least $C^1$ for each $s,z$ (the derivative of $\phi$ wrt to $x$ is continuous with respect to $s$ and $x$ for each fixed $z$) by \cite[Theorem 2.2.2 (iv)]{arnoldbook} and $M$ is compact, $\mathcal{P}_{z}^{(\delta)}f(x)\to f(x)$ uniformly in $x$ as $\delta\to 0$.
Thus $\|\mathcal{P}_{z}^{(\delta)}f-f\|_1\to 0$ as $\delta\to 0$.
Since the continuous functions are dense in $L^p$, $1\le p<\infty$
as $M$ is compact (see e.g.\ \cite{dunford_schwartz} Lemma IV.8.19),
one can $L^1$ approximate any $L^1$ $f$ by a continuous function and
thus the result holds for all $L^1$ functions $f$.
\end{proof}
Thus the result follows using Lemma \ref{basicctylemma} and the fact that $\|\pi_n\|_1=1$.
\end{proof}

Now, let's compare $\hat{f}_{n,\xi(\delta,z),M,N}$ and $\hat{f}_{n,z,M,N}$.
\begin{proposition}
$\|\hat{f}_{n,\xi(\delta,z),M,N}-\hat{f}_{n,z,M,N}\|_1\to 0$ as $\delta\to 0$.
\end{proposition}
\begin{proof}
This result is more difficult to demonstrate as we need to firstly
compare $\Psi^{(M)}(\xi(-N,z))$ with $\Psi^{(M)}(\xi(-N+\delta,z))$. To this end,
consider
\begin{eqnarray*}
\|\pi_n\mathcal{P}_{\xi(-N+\delta,z)}^{(M)}f-\pi_n\mathcal{P}_{\xi(-N,z)}^{(M)}f\|_1&=&\|\pi_n\mathcal{P}_{\xi(\delta,z)-N}^{(M)}f-\pi_n\mathcal{P}_{\xi(-N,z)}^{(M+\delta)}f+\pi_n\mathcal{P}_{\xi(-N,z)}^{(M+\delta)}f -\pi_n\mathcal{P}_{\xi(-N,z)}^{(M)}f\|_1 \\
&\le&\|\pi_n\|_1\left(\|\mathcal{P}_{\xi(-N+\delta,z)}^{(M)}f-\mathcal{P}_{\xi(-N,z)}^{(M+\delta)}f\|_1+\|\mathcal{P}_{\xi(-N,z)}^{(M+\delta)}f -\mathcal{P}_{\xi(-N,z)}^{(M)}f\|_1\right) \\
&\le&\|\mathcal{P}_{\xi(-N+\delta,z)}^{(M)}(\Id-\mathcal{P}_{\xi(-N,z)}^{(\delta)})f\|_1+\|(\mathcal{P}_{\xi(-N,z)+M}^{(\delta)}-\Id)\mathcal{P}_{\xi(-N,z)}^{(M)}f\|_1 \\
&\le&\|(\Id-\mathcal{P}_{\xi(-N,z)}^{(\delta)})f\|_1+\|(\mathcal{P}_{\xi(-N,z)+M}^{(\delta)}-\Id)\mathcal{P}_{\xi(-N,z)}^{(M)}f\|_1 \\
\end{eqnarray*}
The right hand side converges to zero as $\delta\to 0$ by Lemma
\ref{basicctylemma}. This result implies that
$\|P^{(M)}(\xi(-N,z))-P^{(M)}(\xi(-N+\delta,z))\|\to 0$ as $\delta\to 0$ in whatever
matrix norm we choose.  Thus $\|\Psi^{(M)}(\xi(-N,z))
^{2M}-\Psi^{(M)}(\xi(-N+\delta,z))^{2M}\|=\|P^{(M)}(\xi(-N,z))^\top(P^{(M)}(\xi(-N,z))-P^{(M)}(\xi(-N+\delta,z)))+(P^{(M)}(\xi(-N,z))^\top-P^{(M)}(\xi(-N+\delta,z))^\top)P^{(M)}(\xi(-N+\delta,z))\|\to
0$ as $\delta\to 0$.  By standard perturbation results, see
e.g.\ \cite[Theorem II.5.1]{kato}, this implies that eigenvectors
$U^{(M)}_2(z)$ and $U^{(M)}_2(\xi(\delta,z))$ are close for sufficiently small $\delta$.
% This should imply
% that eigenspaces are close for isolated eigenvalues (see eg.
% Chatelin ``Spectral Approximation for Linear Operators'' or Stewart
% and Sun ``Matrix Perturbation Theory''...both are borrowed from the
% library here so I cannot check myself.). INSERT LEMMA ABOUT
% EIGENSPACES of $\Psi^{(M)}(t)$ with $\Psi^{(M)}(\xi(\delta,z))$ BEING
% CLOSE.
Thus $f_{n,\xi(-N,z),M}$ and $f_{n,\xi(-N+\delta,z),M}$ are close in $L^1$ norm.
Now we need to push both of these forward by $\pi_n\mathcal{P}(\xi(-N,z))^{(N)}$.
This will not increase the norm of the difference at all, so $\|\hat{f}_{n,\xi(\delta,z),M,N}-\hat{f}_{n,z,M,N}\|_1$ will also be small.
\end{proof}

\subsection{Oseledets functions for a 1D discrete time nonautonomous system}
\label{discreteeigenmodesection}

We now examine the Oseledets functions for the system defined in
Example \ref{Discreteeg}. We consider the approximation
$\pi_{100}\mathcal{P}_{\omega}$ of rank $100$, which we obtain by
Galerkin projection. We denote by
$P(\omega)\in\R^{100}\times\R^{100}$ the Ulam matrix representing
the action of $\pi_{100}\mathcal{P}_{\omega}$ on functions
$f\in\mathcal{B}_{100}:=\sp\{\chi_{[(i-1)/100,i/100)},
i=1,\ldots,100\}$. The matrices $P(\sigma^{-k}\omega),
k=-10,\ldots,10$ are constructed by following Algorithm 1 using
$Q=100$.

We look for Oseledets functions for a particular aperiodic sequence
$\omega$. To generate an aperiodic sequence, let
$\tau\in\set{0,1}^\N$ be the binary expansion of $1/\sqrt{3}$.
Extend $\tau$ to an element of $\set{0,1}^\Z$ by setting $\tau_i=0$
for all $i\leq 0$. Define $\omega_{i-25}=1+2\tau_i+\tau_{i+1}$ for
each $i$. Then $\omega\in\Omega$ and the central $21$ terms of
$\omega$ are
\begin{equation}
 \label{specialomega}
\omega = (\ldots, 2, 3, 1, 2, 4,  4, 3, 2, 3, 1,  \dot{1},  2, 3, 2, 4, 3,  1, 2, 3, 1, 1  \ldots),
\end{equation}
where the dot denotes the zeroth term $\omega_0=1$.

We calculate the eigenvalues of $(P^{(20)}(\sigma^{-10}\omega)^\top\circ P^{(20)}(\sigma^{-10}\omega))^{1/40}$, where $P^{(20)}(\sigma^{-10}\omega)$ is defined as in (\ref{discretematrixcocycle}), and find the top three to be
$$
L_1\approx 1.00, L_2\approx 0.84, L_3\approx 0.46.
$$
As in Examples \ref{singlemapeg} and \ref{periodiceg}, the maps
$T_i$ are piecewise affine with constant slope three, and so
$\rho(\omega)=1/3$. Thus $\log L_2$ and $\log L_3$ may approximate
isolated Lyapunov exponents in $\Lambda(\mathcal{P})$.

We follow Algorithm 2 to approximate the second Oseledets subspace $W_2^{(M,N)}(\sigma^k\omega)$ for $k=0,\ldots,5$, using $(M,N)=(20,10)$, see Figure \ref{fig:NonCS}.
\begin{figure}[tb]
    \centering
    \psfrag{L0}{$k=0$}
    \psfrag{L1}{$k=1$}
    \psfrag{L2}{$k=2$}
    \psfrag{L3}{$k=3$}
    \psfrag{L4}{$k=4$}
    \psfrag{L5}{$k=5$}
        \includegraphics[width=0.90\textwidth]{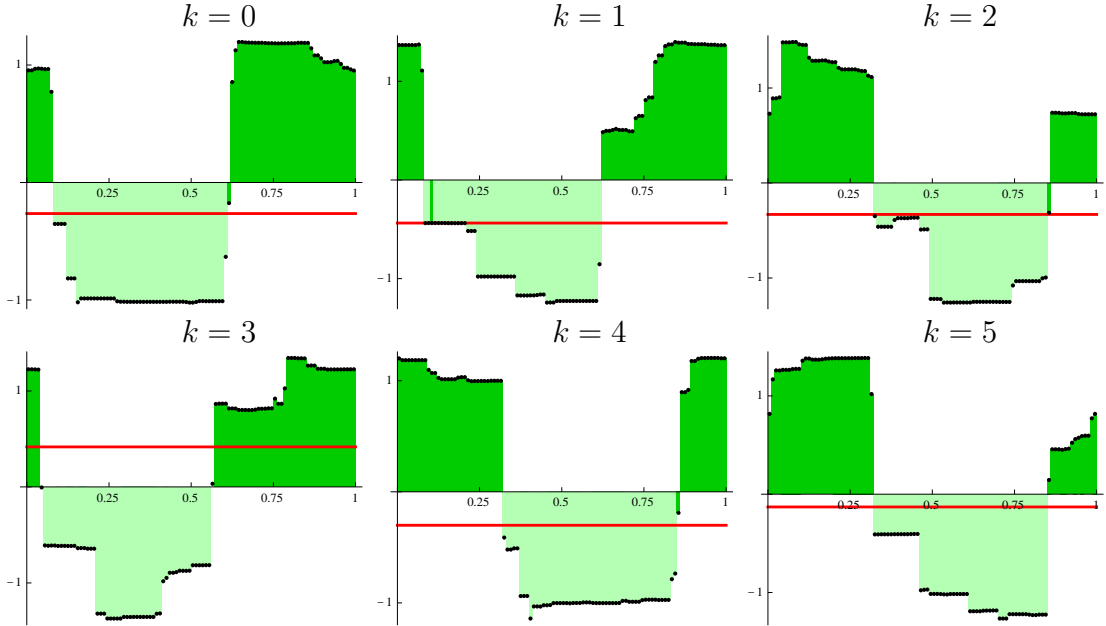}
    \caption{The Oseledets function approximations $f_2^{(M,N)}(\sigma^k\omega)$ for $M=20, N=10$, and $k=0,\ldots,5$, along with optimal thresholds (shown in dashed green), see Section \ref{1dcssection}.}
    \label{fig:NonCS}
\end{figure}
%If we consider the sets $\{f_2^{(20,10)}(\sigma^k\omega)>0\}_{k\in\N}$ with $f_2=\sum_{i=1}^{100}w_2^{(20,10)}(\sigma^k\omega)\chi_{B_i}$, then we see that these sets approximately (but not exactly) correspond to the sets $[0,1/2], [0,1/2], [1/4,3/4],\ldots$ that are approximately ``carried'' by $T_1, T_2, T_4,\ldots$ at times $0, 1, 2,\ldots$ according to $\omega$ in (\ref{specialomega}).  We return to this property again in \S\ref{1dcssection}.
In order to confirm the effectiveness of Algorithm 2 we calculate
the $L^1$ distance $\Delta(N)$ between the normalisations of the vectors
$w_2^{(2N,N)}(\sigma\omega)$ and $P(\omega) w_2^{(2N,N)}(\omega)$,
for $N=2,\ldots,19$ with $M=40$. By property 3 of Theorem \ref{thm:main} this distance should be small if the family $W_2(\omega)$ is well approximated. A logarithmic plot of $\Delta(N)$ against $N$,
see Figure \ref{fig:logplot}, shows the fast convergence of
$w_2^{(2N,N)}(\omega)$ to an Oseledets subspace.
\begin{figure}[bt]
   \centering
   \psfrag{a}{$N$}
   \psfrag{b}{$\Delta(N)$}
       \includegraphics[width=0.70\textwidth]{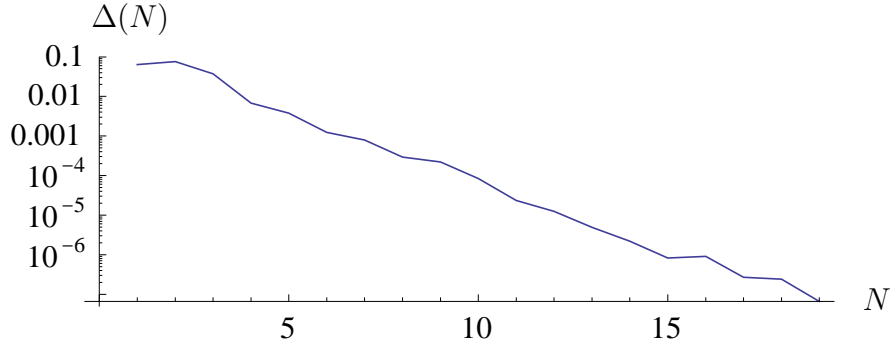}
   \caption{A graph showing $\Delta(N)$ for $N=1,\ldots,19$.}
   \label{fig:logplot}
\end{figure}
%\begin{figure}[bt]
%    \centering
%        \includegraphics[width=0.70\textwidth]{LogPlot.eps}
%    \caption{A graph showing $\log_{10} \Delta(N)$ for $N=2,\ldots,9$.}
%    \label{fig:logplot}
%\end{figure}
In Section \ref{1dcssection} we will see how to extract coherent
sets from these functions.

%
% Random composition of interval maps.
%
% Insert here Simon's example of aperiodic combination of infinite
% family of non-Markov interval maps. Use this example to illustrate
% the Ulam construction and to make some stronger statements about
% formal results.
% \begin{itemize}
% \item Insert figure showing eigenfunctions at time 0 and at least one
% other time.
%
% \item Insert figure showing eigenfunction at time $\tau$ and push
% forward of eigenfunction at time 0 to show they match well.
% \end{itemize}

\subsection{Oseledets functions in a 2D continuous time nonautonomous system}
\label{ctscomputationssect}

We consider the following nonautonomous system on
$M=[0,2\pi]\times[0,\pi]$, $t\in\mathbb{R}^+$:
\begin{equation}\label{eq::travellingwave}
\begin{split}
\dot{x} &=  c-A\sin(x-\nu t)\cos(y)\qquad\mod{2\pi}\\
\dot{y} &=  A\cos(x-\nu t)\sin(y)\\
\end{split}
\end{equation}
This equation describes a travelling wave in a stationary frame of
reference with rigid boundaries at $y=0$ and $y=\pi$, where the
normal flow vanishes~\cite{Pierrehumbert91,SamelsonWiggins}. The
streamfunction (Hamiltonian) of this system is given by
\begin{equation}\label{eq::wavehamiltonian}
s(x,y,t)=-cy+A\sin(x-\nu t)\sin(y).
\end{equation}
We set $c=0.5$, $A=1$, and the phase speed to $\nu=0.25$. The
velocity field is $2\pi$-periodic in the $x$-direction, which allow
us to study the flow on a cylinder. The velocity fields in a
comoving frame for these parameters are shown in
Figure~\ref{fig::vectorfield_nodisturb}. The closed recirculation
regions adjacent to the walls ($y=0$ and $y=\pi$) move in the positive
$x$-direction and are separated from the jet flowing regime by the
heteroclinic loops of fixed points, which are given below.
\par
This model can be simplified to an autonomous system with a steady
streamfunction in the comoving frame by setting $X=x-\nu t$ and
$Y=y$. The steady streamfunction is then given by
$S(X,Y,t)=-(c-\nu)Y+A\sin(X)\sin(Y)$. Let $X_s=\sin^{-1}((c-\nu)/A)$
and $Y_s=\cos^{-1}((c-\nu)/A)$. In the comoving frame, the
recirculation region at the wall $Y=0$ contains an elliptic point
$q_1=(\pi/2,Y_s)$ and is bounded by the heteroclinic loop of the
hyperbolic fixed points $p_1=(X_s,0)$ and $p_2=(\pi-X_s,0)$.
Similarly, those elliptic and hyperbolic points at the wall $Y=\pi$
are $q_2=(3\pi/2,\pi-Y_s)$, $p_3=(\pi+X_s,\pi)$, and
$p_4=(2\pi-X_s,\pi)$, respectively, see
Figure~\ref{fig::vectorfield_nodisturb}. One may observe that there is a continuous family of \emph{invariant} sets %the
%notion of invariant set is indeed ambiguous
in the comoving frame as
%since one could define
any fixed level set of the streamfunction bounds an {invariant} set.
In a stationary frame these elliptic and hyperbolic points (and
their heteroclinic loops) are just translated in the $x-$direction.
That is, any fixed level set of the time-dependent
streamfunction~\eqref{eq::wavehamiltonian} is a (time-dependent)
invariant manifold. We note, however, that the recirculation regions
are distinguished from the remainder of the cylinder as they are
separated from the jet flowing region, which has a different
dynamical fate. In the subsequent sections we will perturb this
somewhat ``degenerate'' system to destroy the continuum of invariant
sets in the comoving frame and produce a small number of
almost-invariant sets (see Definition \ref{aidefn}) in the comoving
frame, or coherent sets in the stationary frame).

\begin{figure}[tb]
\centerline{\includegraphics[scale=0.5]{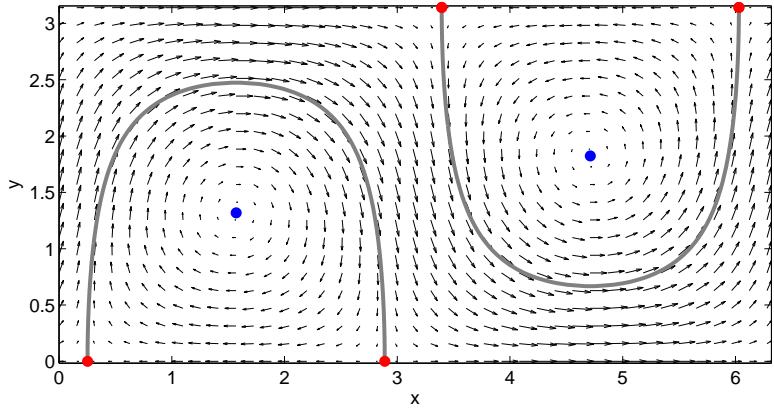}}
\caption{Vector fields in the comoving frame for the travelling wave
flow~\eqref{eq::wavehamiltonian}, for $A=1.0$ and $c=0.5$. The red
dots are the hyperbolic fixed points that are connected by the
heteroclinic loops. The blue dots are elliptic points in the centre
of recirculation regions.} \label{fig::vectorfield_nodisturb}
\end{figure}

\subsubsection{A coherent family: Mixing case}
\label{5.5.1} We modify the traveling wave model in the previous
section to allow mixing in the jet flowing region. We add a
perturbation to the system in the following way:
\begin{equation}\label{eq::travellingwavemixing}
\begin{split}
\dot{x} &=  c-A(\tilde{z}(t))\sin(x-\nu\tilde{z}(t))\cos(y)+\varepsilon G(g(x,y,\tilde{z}(t)))\sin(\tilde{z}(t)/2)\\
\dot{y} &=  A(\tilde{z}(t)))\cos(x-\nu \tilde{z}(t)))\sin(y).\\
\end{split}
\end{equation}
Here, $\tilde{z}(t)=6.6685z_1(t)$, where $z_1(t)$ is generated by the Lorenz flow in
Example~\ref{ex:LZ} with initial point $z(0)=(0,1,1.5)$,  $A(\tilde{z}(t))=1+0.125\sin(\sqrt{5}\tilde{z}(t))$,
$G(\psi):=1/{(\psi^2+1)}^2$ and the parameter function
$\psi=g(x,y,\tilde{z}(t)):=\sin(x-\nu\tilde{z}(t))\sin(y)+y/2-\pi/4$ vanishes at
the level set of the streamfunction of the unperturbed flow at
instantaneous time t=0, i.e., $s(x,y,0)=\pi/4$, which divides the
phase space in half.
%NARATIP, DO NOT UNDERSTAND WHAT FOLLOWING SENTENCE IS TRYING TO SAY.
%\textbf{Although} this level set is not exactly at the
%middle of the jet regime of the perturbed system, the perturbation
%term is \textbf{still} large \textbf{only} in the vicinity of the middle of the jet
%regime \textbf{but} very small near the recirculation region.
We set
$\varepsilon=1$ as this value is sufficiently large to ensure no KAM
tori remain in the jet regime, but sufficiently small to maintain
islands originating from the nested periodic orbits around the
elliptic points of the unperturbed system.

%We note that since $A(t)$ has period $4\pi/3$, $\sin(x-\nu t)$ has period $8\pi$, $g(x,y,t)$ has period $8\pi$, and $\sin(t/2)$ has period $4\pi$, the right hand sides of (\ref{eq::travellingwavemixing}) are both periodic with period $8\pi$.
%Thus, in the terminology introduced in Section 2.2, we may set $\Xi=S^1_{8\pi}$, a circle with circumference $8\pi$, $\xi(t,z)=z+t$, and $\mathbb{P}$ to be normalised Lebesgue measure on $\Xi$, and rewrite (\ref{eq::travellingwavemixing}) as
%\begin{equation}\label{eq::travellingwavemixing2}
%\begin{split}
%\dot{x} &=  c-A(z)\sin(x-\nu z)\cos(y)+\varepsilon G(g(x,y,z))\sin(z/2)\\
%\dot{y} &=  A(z)\cos(x-\nu z)\sin(y)\\
%\dot{z} &=  1.\\
%\end{split}
%\end{equation}
%Using the generalised time $z$ does not alter the fact that our perturbed system is periodically driven.
%However, in our computations below we will compute for less than a full period.
%OOPS, WE COMPUTED FOR A DURATION OF 80 TIME UNITS...APPROX THREE FULL PERIODS!!
%FOR THIS TO BE CONVINCING, WE NEED TO COMPUTE FOR LESS THAN A FULL PERIOD.

We applied Algorithm \ref{ctsPalg} with $n=28800$, $M=80$, $N=40$,
$z=(0,1,1.5)$, and Algorithm \ref{ctseigenmodealg} for $z=(0,1,1.5)$ and $z=\xi(10,(0,1,1.5))$.  By
using a relatively large number of test points per grid box ($n=400$
points per box $B_j$) we are able to flow for $M=80$ units of time
and still well represent $\phi(80,\xi(-40,z),B_j)$.  Figure
\ref{fig::chaoticpushforward} shows that the resulting Oseledets
functions highlight the remaining islands in the perturbed
time-dependent flow. We calculate the eigenvalues of
$(P^{(80)}(\xi(-40,z))^\top\circ P^{(80)}(\xi(-40,z))^{1/2}$,
where $P^{(80)}(\xi(-40,z))$ is defined as in
(\ref{continuousapproxulam}), and find the top three to be
$$
L_1\approx 1.1100, L_2\approx 0.9691, L_3\approx 0.9676.
$$

\begin{figure}[tb]
\centerline{\includegraphics[scale=0.6]{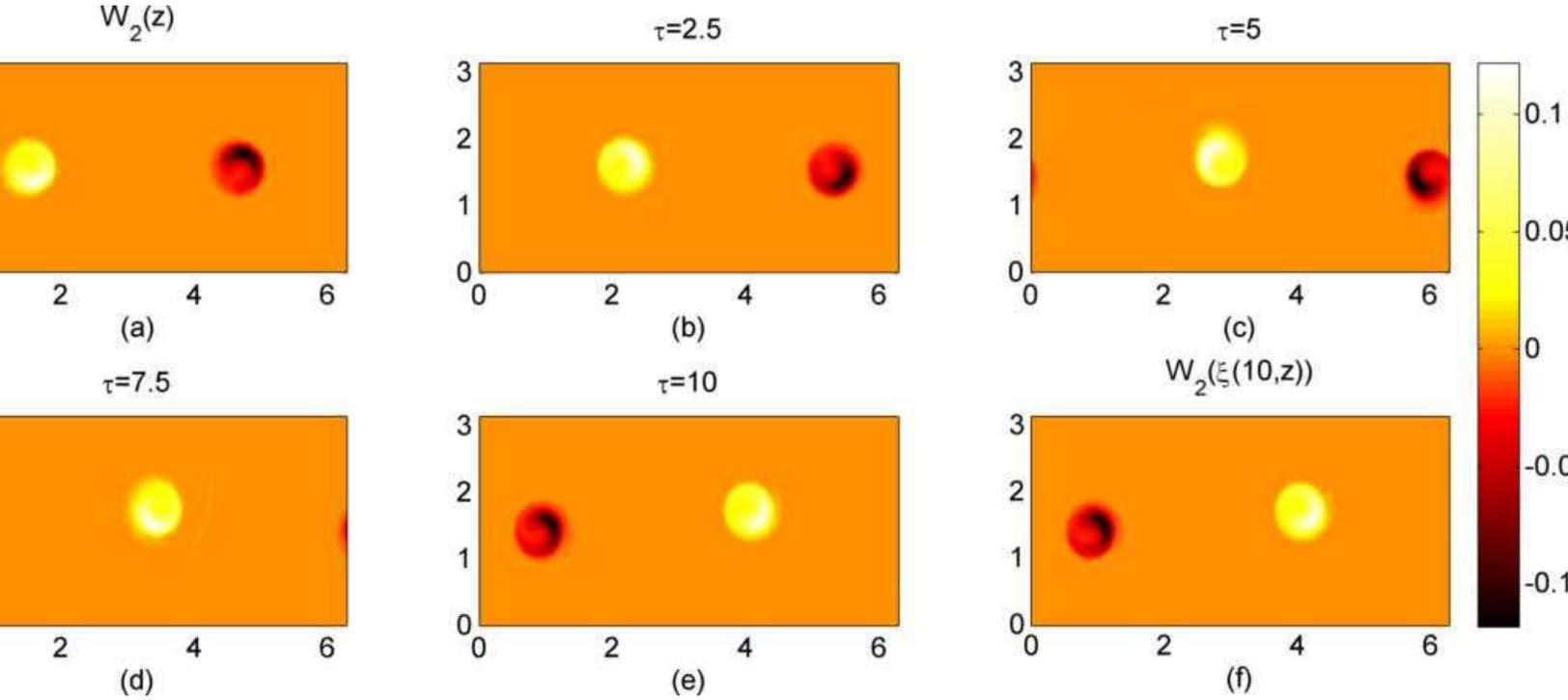}}
\caption{(a) Graph of approximate Oseledets function
$W_2^{(80,40)}(z)$ produced by Algorithm
\ref{ctseigenmodealg}. (b)-(e) Pushforwards of
$W_2^{(80,40)}(z)$ via multiplication by
$P^{(\tau)}(z)$ for $\tau={2.5, 5, 7.5, 10}$. (f)
$W_2^{(80,40)}(\xi(10,z))$ produced independently by
Algorithm \ref{ctseigenmodealg}; compare with (e)}
\label{fig::chaoticpushforward}
\end{figure}
By part 3 of Theorem \ref{thm:main} (bundle invariance of
$W_2(z)$) we should have
$P^{(10)}(z)W^{(80,40)}_2(z)\approx
W^{(80,40)}_2(\xi(10,z))$. This is demonstrated in Figure
\ref{fig::chaoticpushforward} by comparing subplots (e) and (f). In
Section \ref{sec:ctsCS} we will see how to extract coherent sets
from these Oseledets functions.

\section{Invariant Sets, Almost-Invariant Sets, and Coherent Sets}
\label{csdefnsection}

 We begin by briefly recounting some of the
background relevant to almost-invariant sets.  If $\Phi$ (resp.\
$\phi$) is autonomous, then $\Omega$ (resp. $\Xi$) consists of a
single point, and we may write $\Phi(-1,\omega,x)=\Phi(-1,x)$
(resp.\ $\phi(-t,z,x)=\phi(-t,x)$).
\begin{definition}\label{invdefn} In the autonomous setting, we call
$A$ an \emph{invariant set} if $\Phi(-1,A)=A$ (resp. \
$\phi(-t,A)=A$ for all $t\ge 0$).
\end{definition}
 The
following definition generalises invariant sets to
\emph{almost-invariant sets}. In the continuous time case we define:
\begin{definition}
\label{aidefn} Let $\mu$ be preserved by the autonomous flow $\phi$.
We will say that a set $A\subset M$ is
\emph{$\rho_0$-almost-invariant} over the interval $[0,\tau]$ if
\begin{enumerate}
 \item
\begin{equation}
\label{rhomu}
\rho_{\mu,\tau}(A):=\frac{\mu(A\cap\phi(-s,A))}{\mu(A)}\ge \rho_0
\end{equation}
for all $s\in [0,\tau]$,
\item $A$ is connected.
\end{enumerate}
\end{definition}
If $A\subset M$ is \emph{almost-invariant} over the interval
$[0,\tau]$, then for each $s\in[0,\tau]$, the probability (according to $\mu$) of a
trajectory leaving $A$ at some time in $[0,s]$, and not returning
to $A$ at time $s$ is relatively small.
%Thus, almost-invariant
%sets can be considered as an alternative type of coherent structures
%in flows.
In the discrete time setting, $\tau=1$, and the obvious changes are
made in Definition \ref{aidefn}. By convention we ask that $A$ is
connected; if $A$ is not connected, we consider each connected
component to be an almost-invariant set for suitable $\rho_0$.

%We define relevant objects in the discrete time case;  the objects
%in the continuous time case are completely analogous.
We now begin to discuss the nonautonomous setting.
The notion of an invariant set is extended to an \textit{invariant family}.
\begin{definition}
\label{invariantset}\quad \begin{enumerate} \item \textbf{Discrete
time:} We will call a family of sets $\{A_{\sigma^k\omega}\}$,
$A_{\sigma^k\omega}\subset M$, $\omega\in\Omega$, $k\in\mathbb{Z}$ an \emph{invariant family} if
$\Phi(-k,\omega,A_\omega)=A_{\sigma^{-k}\omega}$ for all
$\omega\in\Omega$ and $k\in \mathbb{Z}^+$.
\item \textbf{Continuous time:} We will call a family of sets $\{A_{\xi(t,z)}\}$, $A_{\xi(t,z)}\subset M$,
$z\in\Xi$, $t\in\mathbb{R}$ an \emph{invariant family} if
$\phi(-t,z,A_z)=A_{\xi(-t,z)}$ for all $z\in \Xi$ and $t\in
\mathbb{R}^+$.
\end{enumerate}
\end{definition}

Motivated by a model of fluid flow, we imagine coherent sets as a
family of \textit{connected} sets with the property that the set
$A_\omega$ is \textit{approximately} mapped onto
$A_{\sigma^k\omega}$ by $k$ iterations of the cocycle from ``time''
$\omega$;  that is, $\Phi(k,\omega,A_\omega)\approx
A_{\sigma^k\omega}$. The definition of coherent sets combines the
properties of almost-invariant sets and an invariant family. As we
now have a \emph{family} of sets we require one more property beyond
those of Definition \ref{aidefn}, in addition to modifying the
almost-invariance property.  In the continuous time case we define:
\begin{definition}
\label{csdefn} Let $\mu$ be preserved by a flow $\phi$ and $0\le \rho_0\le 1$. Fix a
$z\in\Xi$. We will say that a family $\{A_{\xi(t,z)}\}_{t\ge 0}$,
$A_{\xi(t,z)}\subset M$, $t\ge 0$ is \emph{a family of
$\rho_0$-coherent sets} over the interval $[0,\tau]$ if:
\begin{enumerate}
\item
\begin{equation}
\label{rhomucts}
\rho_{\mu}(A_{\xi(t,z)},A_{\xi(t+s,z)}):=\frac{\mu(A_{\xi(t,z)}\cap\phi(-s,\xi(t+s,z),A_{\xi(t+s,z)}))}{\mu(A_{\xi(t,z)})}\ge
\rho_0,
\end{equation}
for all $s\in [0,\tau]$ and $t\ge 0$,
\item Each $A_{\xi(t,z)}$, $t\ge 0$ is connected,
\item $\mu(A_{\xi(t,z)})=\mu(A_{\xi(t',z)})$ for all $t,t'\ge 0$,
%\item
%%\begin{itemize}
% %\item
%$\vol_{d-1}\partial A_{\xi(t,z)}\le C$ for $t\ge 0$, where
%$\vol_{d-1}$ is $d-1$-dimensional volume.  We define $\vol_0$ to be
%the function that counts the number of boundary points.
%\item There exists a $B\subset M$ and $\Delta>0$ so that $|\mu(\phi(-t,z,A_{\xi(t,z)}\cap B)-\mu(A_{\xi(0,z)})\mu(B)|\ge \Delta$ for all $t\ge 0$.
%\end{itemize}
%\item ??A local optimality condition on the family in terms of $\rho$??
\end{enumerate}
\end{definition}
In discrete time, we replace (\ref{rhomucts}) with
\begin{equation}
\label{rhomudiscrete}
\rho_{\mu}(A_\omega):=\frac{\mu(A_\omega\cap\Phi(-1,\sigma\omega,A_{\sigma\omega}))}{\mu(A_\omega)}\ge \rho_0,
\end{equation}
$\tau$ necessarily becomes 1,
and we make the obvious changes to the other items in  Definition  \ref{csdefn}.

We remark that by selecting some $A\subset M$ of positive $\mu$
measure and defining $A_{\xi(t,z)}:=\phi(t,z,A)$, $t\ge 0$, the
family $\{A_{\xi(t,z)}\}_{t\ge 0}$, is a family of 1-coherent sets.
Such a family is not of much dynamical interest, as there is nothing
distinguishing this family from one constructed with another
connected subset $A'\subset M$. We are not interested in these
constructions of coherent sets, and in practice the numerical
algorithm we present in the next section is unlikely to find such
sets for chaotic systems. %We return to this issue after presenting

\section{Coherent sets from Oseledets functions}
\label{cssection}
%THIS SECTION NEEDS WRITING! Recall that our discretised eigenmodes
%$f$ will be of the form $f=\sum_{i=1}^n w_i\chi_{B_i}$, where the
%vector $w\in\mathbb{R}^n$ lies in an Oseledets subspace for our
%random matrix product. In order to extract coherent structures from
%a family of strange eigenmodes we modify a heuristic that has
%previously been used to extract almost-invariant sets from
%eigenfunctions of Perron--Frobenius operators; see
%\cite{FD03,F05,FP08}.

We wish to find a family of sets $\{A_{z}\}$ so that
\begin{equation} \label{rhomuctsdefn}
\rho_{\mu}(A_{z},A_{\xi(s,z)}):=\frac{\mu(A_{z}\cap\phi(-s,\xi(s,z),A_{\xi(s,z)}))}{\mu(A_{z})}
\end{equation}
is large for $s\in[0,\tau]$.
%Firstly, item 1.\ of Definition \ref{csdefn} needs some more
%explanation.
We may rewrite the RHS of (\ref{rhomuctsdefn})
%item 1.\ of Definition
%\ref{csdefn}
as
\begin{eqnarray}
\label{eqn3}\left(\int \chi_{A_z}\cdot\chi_{\phi(-s,A_{\xi(s,z)})}\,\mathrm{d}\mu\right)/\mu(A_z)
%&=&\left(\int \chi_{A_t}\cdot\chi_{A_{\xi(s,z)}}
%\circ\phi(s,t)\ d\mu\right)/\mu(A_t)\\
&=&\left(\int
\mathcal{P}_z^{(s)}\chi_{A_z}\cdot\chi_{A_{\xi(s,z)}} \,\mathrm{d}\mu\right)/\mu(A_z).
\end{eqnarray}
For (\ref{eqn3}) to be large we require
$\mathcal{P}_z^{(s)}\chi_{A_z}\approx \chi_{A_{\xi(s,z)}}$.

 Let us now
make a connection with the Oseledets functions $f_2(z)=\sum_{i=1}^n
w_{2,i}(z)\chi_{B_i}$ where $w_2(z)\in W_2^{(M,N)}(z)$ obtained in Algorithm
\ref{ctseigenmodealg}.
%As discussed in Remark
%\ref{uniqueremark}, if the second largest element $\lambda_2$ of
%$\Lambda(P,t)$ has unit multiplicity, then the family
%$\{f_2(t+k)\}_{k\ge 0}$ represents the unique (up to scalar
%multiples) family of functions that decay
% at rate $\lambda_2$ and satisfy the bundle invariance
% $\pi_n\mathcal{P}f_2(t+k)=\alpha f_2(t+k+1)$.
In the following discussion, we scale $f_2(z)$ so that
$\|f_2(z)\|_1=1$ for all $z\in\Xi$. To convert the family of
Oseledets functions into a family of coherent sets, we modify a
heuristic due to \cite{dellnitz_junge_99} that has been successfully
used in the autonomous setting. The heuristic is to set
$A_z=\{f_2(z)>0\}$, $z\in\Xi$.
%We now discuss the rationale for this
%heuristic, referring to the three elements of Definition
%\ref{csdefn}.
We show that $\mathcal{P}_z^{(s)}f^+_2(z)-f^+_2(\xi(s,z))$ is small for
moderate $s$ and large $\lambda_2$;  we then heuristically infer that
$\mathcal{P}_z^{(s)}\chi_{A_z}\approx \chi_{A_{\xi(s,z)}}$.
%Define
%$\lambda_2(s,t)$ by
%$\mathcal{P}_t^{(s)}f_2(t)=e^{s\lambda_2(s,t)}f_2(\xi(s,z))$.

\begin{proposition}
\label{heurprop}
%\begin{enumerate}\quad
%\item
Let $\lambda_2=\lim_{s\to\infty} (1/s)\log\|\mathcal{P}_z^{(s)}f_2(z)\|<0$ be the second largest Lyapunov exponent from
Theorem \ref{thm:main} and $f_2(z)\in W_2(z)$ a corresponding Oseledets
function, normalised so that $\|f_2(z)\|_1=1$. Given an $\epsilon>0$ there is an $S\ge 0$ so that $s\ge
S$ implies $\|\mathcal{P}_z^{(s)}f^+_2(z)-f^+_2(\xi(s,z))\|_1\le
(1-e^{(\lambda_2-\epsilon)s})/2$.
%\item $A_{\xi(s,z)}:=\{f^+_2(\xi(s,z))>0\}\subset \{\mathcal{P}_t^sf^+_2(t)>0\}=:\phi^s(t)A_t$.
%\end{enumerate}
\end{proposition}
\begin{proof}
%\begin{enumerate}\item
 Given
$\epsilon>0$ we know that there exists $S\ge 0$ such that for all $s\ge
S$ one has
$e^{\lambda_2-\epsilon}\le\|\mathcal{P}_z^{(s)}f_2(z)\|^{1/s}\le 1$.
Since
$\mathcal{P}_z^{(s)}f_2(z)=(\mathcal{P}_z^{(s)}f_2(z))^+-(\mathcal{P}_z^{(s)}f_2(z))^-$
and $\int \mathcal{P}_z^{(s)}f_2(z)\,\mathrm{d}m=0$, one has
$\|\mathcal{P}_z^{(s)}f_2(z)\|_1=\int
(\mathcal{P}_z^{(s)}f_2(z))^++(\mathcal{P}_z^{(s)}f_2(z))^-\,\mathrm{d}m=2\int
(\mathcal{P}_z^{(s)}f_2(z))^+\,\mathrm{d}m$. Thus $\int
(\mathcal{P}_z^{(s)}f_2(z))^+\,\mathrm{d}m\ge e^{(\lambda_2-\epsilon)s}/2$.
Since $(\mathcal{P}_z^{(s)}f_2(z))^+\le \mathcal{P}_z^{(s)}f^+_2(z)$, one
has $\|\mathcal{P}_z^{(s)}f^+_2(z)-(\mathcal{P}_z^{(s)}f_2(z))^+\|=\int
\mathcal{P}_z^{(s)}f^+_2(z)-(\mathcal{P}_z^{(s)}f_2(z))^+\,\mathrm{d}m$. As
$\|f_2(z)\|=1$ and $\int f_2(z)\,\mathrm{d}m=0$, we have $\int f^+_2(z)\,\mathrm{d}m=1/2$ and since
$\mathcal{P}_z^{(s)}$ preserves integrals, $\int
\mathcal{P}_z^{(s)}f^+_2(z)\,\mathrm{d}m=1/2$. Thus, $\int
\mathcal{P}_z^{(s)}f^+_2(z)-(\mathcal{P}_z^{(s)}f^+_2(z))\,\mathrm{d}m\le
(1-e^{(\lambda_2-\epsilon)s})/2$. %\item Since
%\end{enumerate}
\end{proof}
The preceding discussion heuristically addresses item 1.\ of
Definition \ref{csdefn}.
%****PROBABLY TO BE DELETED FROM HERE****
% Given $\epsilon>0$, we can find
%$S>0$ such that for $s\ge S$,
%$\exp(\lambda_2-\epsilon)\le\|\mathcal{P}^{(s)}f_2(t)\|^{1/s}\ge
%\exp(\lambda_2+\epsilon)$.  For $\lambda_2$ close to 0, one obtains
%$\|\mathcal{P}^{(s)}f_2(t)\|^{1/s}\approx 1$ (*).  Note that any
%family of Oseledets functions must satisfy $\int f_2(t)(x)\
%d\mu(x)=0$\footnote{since otherwise if for some $t$, $\int
%f_2(t)(x)\ d\mu(x)>0$, because $\mathcal{P}$ preserves integrals,
%all future $f_2(t)$ would also have positive integral and
%$\lambda_2=0$.}. Consider $f_2^+(t):=\max\{f_2(t),0\}$. By (*),
%$\mathcal{P}^{(s)}f_2^+(t)\approx (\mathcal{P}^{(s)}f_2(t))^+$. But
%$(\mathcal{P}^{(s)}f_2(t))^+=\alpha f_2^+(\xi(s,z))$ for some
%$\alpha\approx \lambda^s_2$. Thus
%\begin{quote}
%$\mathcal{P}^{(s)}f_2^+(t)$ is approximately a scaled version of
%$f_2^+(\xi(s,z))$. (**)
%\end{quote}
%
%Setting $A_t=\{f_2(t)>0\}$ and using property (**) we may expect the
%integral (\ref{eqn3}) (and thus (\ref{eqn1})) to be large
%\textit{for all $t\ge 0$}. ***TO HERE***
Regarding item 2 of Definition \ref{csdefn}, as we are extracting
the sets $A_z$ from the Oseledets functions $f_2(z)$, the
connectivity of the sets will depend on the regularity of the
Oseledets functions. This is a delicate question and relatively
little can be said formally at present. In the autonomous case,
roughly speaking, one expects smooth eigenfunctions for
Perron--Frobenius operators of smooth expanding systems
\cite{keller_84,ruelle_82}, and eigendistributions (smooth in
expanding directions, distributions in contracting directions) in
uniformly hyperbolic settings \cite{bkl}.  These properties may
carry over to the non-autonomous setting;  recent results in the bounded variation setting show they do \cite{froyland_lloyd_quas2}.  If a small amount of
noise is added by postmultiplying the Perron--Frobenius operator by a
smoothing (e.g.\ diffusion) operator, then the Oseledets functions must be
smooth. This physical addition of a small amount of noise is one way
to guarantee regularity of the Oseledets functions and connectivity
of the associated coherent sets.

Finally we note that if $\mu=m$  one has $\int f_2(z)(x)\,\mathrm{d}\mu(x)=0$
and so we must have $\mu(A_z)=1/2$ for all $z\in\Xi$. Thus, item 3.\
of Definition \ref{csdefn} is satisfied by the choice
$A_z=\{f_2(z)>0\}$.  If $\mu\neq m$, then it may be necessary to
further tweak the choice of the $A_z$ to ensure that item 3.\ of
Definition \ref{csdefn} is satisfied.  This additional tweak is
described in Algorithm \ref{csalg}.

\subsection{A numerical algorithm}
For a fixed time
$z\in\Xi$, we seek to approximate a pair of sets $A_{z}$ and
$A_{\xi(\tau,z)}$ for which
\begin{equation} \label{rhomuctsdefn2}
\rho_{\mu}(A_{z},A_{\xi(\tau,z)}):=\frac{\mu(A_{z}\cap\phi(-\tau,\xi(\tau,z),A_{\xi(\tau,z)}))}{\mu(A_{z})}
\end{equation}
is maximal. The quantity $\rho_{\mu}(A_{z},A_{\xi(\tau,z)})$ is simply
the fraction of $\mu$-measure of $A_z$ that is covered by a pullback
of the set $A_{\xi(\tau,z)}$ over a duration of $\tau$.  For maximal
coherence, we wish to find pairs $A_z$, $A_{\xi(\tau,z)}$ that maximise
$\rho_{\mu}(A_{z},A_{\xi(\tau,z)})$.
%However, the family $\{A_t\}_{t\ge
%0}$ from which we draw these sets should satisfy the additional
%conditions of Definition \ref{csdefn}.
We present a heuristic to find such a pair of sets based upon the
vectors $W^{(M,N)}(z)$ and $W^{(M,N)}(\xi(\tau,z))$ corresponding to some Lyapunov spectral value $\lambda$ close to 0. This heuristic is a
modification of heuristics to determine maximal almost-invariant
sets, see
\cite{froyland_dellnitz_03,froyland_05,froyland_padberg_08}. In the
terminology of the prior discussion in \S7, rather than setting
$A_z:=\{f(z)>0\}$, we allow $A_z:=\{f(z)>c\}$ or
$A_z:=\{f(z)<c\}$ for some $c\in\mathbb{R}$ in the hope of finding
$A_z, A_{\xi(\tau,z)}$ with an even greater value of
$\rho_{\mu}(A_{z},A_{\xi(\tau,z)})$.  This additional flexibility also
permits a matching of $\mu(A_z)$ and $\mu(A_{\xi(\tau,z)})$.

\begin{algorithm}[To determine a pair of maximally coherent sets at
times $z, \xi(\tau,z)$] \label{csalg}
\begin{enumerate}
\quad
\item Determine $W^{(M,N)}(z)$ and $W^{(M,N)}(\xi(\tau,z))$ for some $\tau>0$ according to Algorithm \ref{ctseigenmodealg}.
\item Set $\hat{A}^+_z(c)=\bigcup_{i:W^{(M,N)}(z)>c}B_i$ and
$\hat{A}^+_{\xi(\tau,z)}(c)=\bigcup_{i:W^{(M,N)}(\xi(\tau,z))>c}B_i$,
restricting the values of $c$ so that $\mu(\hat{A}^+_z(c)),
\mu(\hat{A}^+_{\xi(\tau,z)}(c))\le 1/2$.  These are sets constructed from
grid boxes whose corresponding entry in the $W^{(M,N)}$ vectors is
above a certain value.
\item Define $\eta(c)=\argmin_{c'\in\mathbb{R}}
|\mu(\hat{A}^+_z(c))-\mu(\hat{A}^+_{\xi(\tau,z)}(c'))|$.  Given a value
of $c$, $\eta(c)$ determines the set $\hat{A}^+_{\xi(\tau,z)}(\eta(c))$
that best matches the $\mu$-measure of $\hat{A}^+_z(c)$, as required
by item 3 of Definition \ref{csdefn}.
\item Set $c^*=\argmax_{c\in \mathbb{R}}
\rho_\mu(\hat{A}^+_z(c),\hat{A}^+_{\xi(\tau,z)}(\eta(c)))$.  The value of
$c^*$ is selected to maximise the coherence.
\item Define $A_z:=\hat{A}^+_z(c^*)$ and
$A_{\xi(\tau,z)}:=\hat{A}^+_{\xi(\tau,z)}(\eta(c^*))$.
\end{enumerate}
\end{algorithm}

\begin{remark}
\label{pmremark} \begin{enumerate} \quad \item One can repeat
Algorithm \ref{csalg}, replacing $\hat{A}^+_z(c)$ and $\hat{A}^+_{\xi(\tau,z)}(c)$ with
$\hat{A}^-_z(c)=\bigcup_{i:W^{(M,N)}(t)<c}B_i$ and
$\hat{A}^-_{\xi(\tau,z)}(c)=\bigcup_{i:W^{(M,N)}(\xi(\tau,z))<c}B_i$ respectively. See
\cite{froyland_padberg_08} for further details.
\item Care should be taken regarding the sign of $W^{(M,N)}(z)$ and
$W^{(M,N)}(\xi(\tau,z))$.  Visual inspection may be required in order to
check that the vectors have the same ``parity''.
\end{enumerate}
\end{remark}

%Describe here how to do a line search of the second largest
%Oseledets vector to find optimal coherent structures.
%
%We need to convert the Oseledets vectors into a coherent family of {\em
%sets}.  How to do this? Here is a suggestion in analogy to the
%procedure for optimising the choice of $c$ from the Froyland/Padberg
%paper. Given $t,\tau$ we wish to define $A_t$, $A_{t-\tau}$ via
%level sets of eigenvectors. The difficulty is that we have two
%eigenvectors;  one at time $t$ ($v_t$) and one at time $t-\tau$
%($v_{t-\tau}$). We thus seem to have two degrees of freedom:  a
%level set choice at $t$ and another level set choice at $t-\tau$.
%However, note that $\mu(A_t)=\mu(A_{t-\tau})$ by definition. Thus we
%should select a level set at say $t-\tau$ to define
%$A_{t-\tau}=\bigcup_{\{i: v_{t,i}>c\}} B_i$, and for such a choice
%of $A_{t-\tau}$ we select $A_t=\bigcup_{\{i: v_{t,i}>c'\}} B_i$
%where $c'$ is chosen so that $\mu(A_t)$ is as close to
%$\mu(A_{t-\tau})$ as possible. Care must be taken to ensure that
%$v_t$ and $v_{t-\tau}$ have the same sign. Now for such $A_t,
%A_{t-\tau}$ one computes $\rho_{\mu,t,\tau}(\{A_t\})$, varying $c$
%so as to maximise this quantity.

\subsection{Coherent Sets for a 1D discrete time nonautonomous system}
\label{1dcssection}

We return to the map cocycle $\Phi$ and Perron--Frobenius cocycle
described in Example \ref{Discreteeg} and identify coherent sets. We
use two methods: firstly, inspection of the composition of maps as
perturbations of maps with invariant sets, and secondly using the
general method of Algorithm 5.

The map cocycle $\Phi$ is defined in terms of a map $H_a$ which has
an almost-invariant set, and this gives rise to a family of coherent
sets in the following way. Recall the definitions of the maps $T_i$,
$i=1,\ldots,4$ and shift space $\Omega$ determined by the adjacency
matrix $B$. The maps $T_i$ have the property that if $B_{i,j}=1$,
then any inner $R$ factors cancel in $T_j\circ T_i$. More generally,
for any $\omega\in\Omega$, we have cancellation of all intermediate
$R$ factors:
\begin{eqnarray}\label{eq:cancelR}
\Phi(k,\omega,\cdot) = R^s\circ
H_{a_{(\sigma^{k-1}\omega)_0}}\circ\cdots\circ H_{a_{\omega_0}}\circ
R^{-t},
\end{eqnarray}
where $s,t\in\set{0,1}$ are given by
$$
s(\omega,k) = \left\{\begin{array}{ll}
0, &\omega_{k-1}\, \textrm{odd},\\
1, &\omega_{k-1}\, \textrm{even},
\end{array}\right.
\, \textrm{and}\ \ t(\omega,k) = \left\{\begin{array}{ll}
0, &\omega_{0} \leq 2\\
1, &\omega_{0} >2.
\end{array}\right.
$$
For the map $H_0$, the interval $[0,0.5]$ is invariant. Moreover,
$[0,0.5]$ is almost-invariant for $H_a$ with $\rho_{\mu}([0,0.5])=
1-2a$. By (\ref{eq:cancelR}), we if we set
\begin{equation}
\label{heurA} \tilde{A}_{\sigma^k\omega}=R^{s(\omega,k)}([0,0.5]),\quad
\mbox{ for each } k\in\N, \end{equation}
 then
$$\rho_\mu(\tilde{A}_{\sigma^k\omega},\tilde{A}_{\sigma^{k+1}\omega})=1-2a_{\omega_k}.
$$
Thus $\set{\tilde{A}_{\sigma^k\omega}}_{k\in\N}$ is a family of
$\rho_0$-coherent sets with
$\rho_0=1-2\max\set{a_1,\ldots,a_4}=0.843$. In the same way, the
invariant set $[0.5,1]$ of $H_0$ leads to a family
$\set{R^{s(\omega,k)}([0.5,1])}_{k\in\N}$ of $\rho_0$-coherent sets
with the same $\rho_0$.

In order to demonstrate the methods of this article, we now show how
Algorithm \ref{csalg} can be used.  We may use the Oseledets
subspaces computed in Section \ref{discreteeigenmodesection} to find
a family of coherent sets. First we apply Algorithm \ref{csalg} to
find a coherent set for the time step $k=0$ to $k=1$. We calculate
$\rho_\mu\left( \hat{A}^+_{\omega}(c), \hat{A}^+_{\sigma\omega}(c) \right)$ as $c$
varies over the elements of the vector $f_2^{(20,10)}(\omega)$; see
Figure \ref{fig:NonFirst} (left). The maximum value of
$\rho_\mu\left(\hat{A}^+_\omega(c),\hat{A}^+_{\sigma\omega}(\eta(c))\right)$ is $0.890$.
The set $A_\omega$ is found to be the interval $[0.11,0.58]$ of
length $\mu(A_\omega)=0.47$; see Figure \ref{fig:NonFirst} (right).
\begin{figure}[htb]
    \centering
    \psfrag{b}{$\rho_\mu$}
    \psfrag{a}{$c$}
    \psfrag{L0}[B][br]{$f_2^{(20,10)}(\omega)$}
        \includegraphics[width=1.0\textwidth]{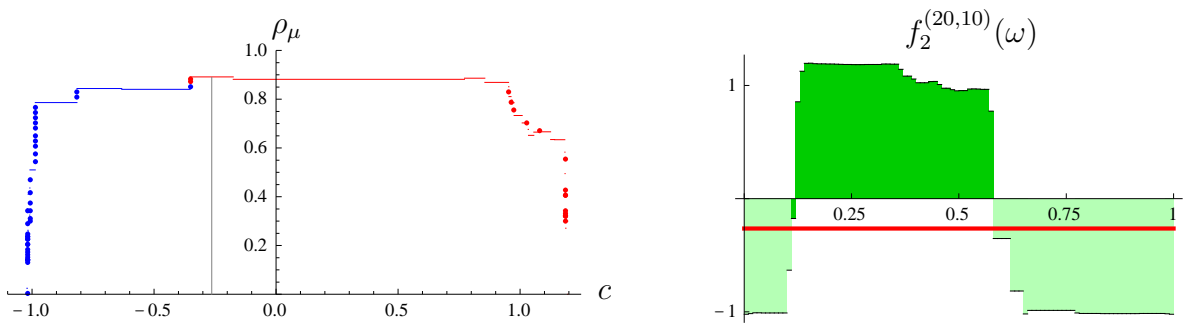}
    \caption{(left): The function $\rho_\mu\left(\hat{A}^+_\omega(c),\hat{A}^+_{\sigma\omega}(\eta(c))\right)$ takes its maximum on the interval $(-0.352,-0.176)$ and so we take the midpoint $c^*=-0.264$ as the optimal threshold. (right): Taking this optimal threshold (shown in dashed green) for the eigenvector
    $f_2^{(20,10)}(\omega)$ identifies the coherent set $A_\omega=[0.11,0.58]$ (shown in dark orange).}
    \label{fig:NonFirst}
\end{figure}

We note that the set $A_\omega$ found by Algorithm \ref{csalg} is
\emph{not} the same as the $\tilde{A}_\omega$ produced by the intuitive
construction (\ref{heurA}). In the latter case, $\tilde{A}_\omega=[0,1/2]$,
$\tilde{A}_{\sigma\omega}=[0,1/2]$, and
$\rho_\mu(\tilde{A}_\omega,\tilde{A}_{\sigma\omega})=1-2a_{\omega_0}=1-2a_1=1-\pi/20\approx
0.843$, significantly lower than the value of 0.890 found using
Algorithm \ref{csalg}.

We may extend Algorithm \ref{csalg} in order to find a sequence of
coherent sets $\{A_{\sigma^i\omega}\}_{i=0}^K$. Since we require the
measure of a sequence of coherent sets to be constant, we seek to
maximize the mean value of $\rho_\mu$ over a given time range as we
vary the measure of the sets.

\begin{algorithm}[To determine a sequence of maximally coherent sets over a range of
times $\omega,\ldots, \sigma^K\omega$] \label{csseqalg}
\begin{enumerate}
\quad
\item Follow steps 1.-3. of Algorithm \ref{csalg}  for each $k=0,\ldots,K-1$ using $\tau=1$ to obtain sets $\hat{A}^+_{\sigma^k\omega}(c)$.
\item Let $c_k(\ell):=\argmin_{c\in\mathbb{R}} \left| \mu(\hat{A}^+_{\sigma^k\omega}(c)) - \ell \right|$.
\item Compute $\ell^*:=\argmax_{\ell\in(0,0.5]} \frac{1}{K}\sum_{k=0}^{K-1} \rho_\mu\left( \hat{A}^+_{\sigma^k\omega}(c_k(\ell)), \hat{A}^+_{\sigma^{k+1}\omega}(c_{k+1}(\ell)) \right)$.
\item For $k=0,\ldots,K-1$, define $A_{\sigma^k\omega}:= \hat{A}^+_{\sigma^k\omega}(c_k(\ell^*))$.
\end{enumerate}
\end{algorithm}

To demonstrate Algorithm \ref{csseqalg}, we use the approximate
Oseledets functions $f_2(\sigma^k\omega)$, $k=0,\ldots,5$, to find a
sequence of six coherent sets $\{A_{\sigma^k\omega}\}_{k=0}^5$ for
the map cocycle $\Phi$. Plotting
$\frac{1}{6}\sum_{k=0}^5\rho_\mu\left(
\hat{A}^\pm_{\sigma^k\omega}(c_k(\ell)),
\hat{A}^\pm_{\sigma^{k+1}\omega}(c_{k+1}(\ell)) \right)$ against
$\ell$ (see Figure \ref{fig:RhoMean}), we find a unique maximum of
$0.891$, which occurs at $\ell^*=0.47$.

%\begin{figure}[htb]
%   \centering
%       \psfrag{1}{$k=0$}
%       \psfrag{2}{$1$}
%       \psfrag{3}{$2$}
%       \psfrag{4}{$3$}
%       \psfrag{5}{$4$}
%       \psfrag{6}{$5$}
%       \includegraphics[width=0.90\textwidth]{RhoEll.eps}
%   \caption{The graph of $\rho_\mu\left( \hat{A}^+_{\sigma^k\omega}(c_k(\ell)), \hat{A}^+_{\sigma^{k+1}\omega}(c_{k+1}(\ell)) \right)$ against $\ell$ for $k=0,\ldots,5$.}
%   \label{fig:RhoEll}
%\end{figure}

\begin{figure}[htb]
    \centering
        \psfrag{rho}{$\overline{\rho_\mu}$}
        \psfrag{ell}{$\ell$}
        \includegraphics[width=0.60\textwidth]{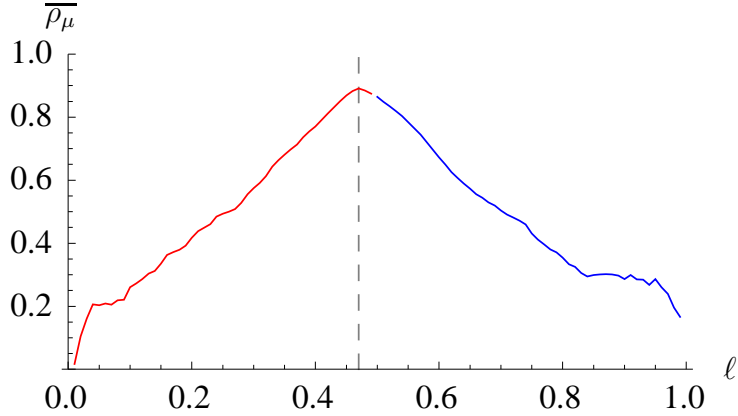}
        \caption{The graph of $\overline{\rho_\mu}:=\frac{1}{6}\sum_{k=0}^5\rho_\mu\left( \hat{A}^\pm_{\sigma^k\omega}(c_k(\ell)), \hat{A}^\pm_{\sigma^{k+1}\omega}(c_{k+1}(\ell)) \right)$ against $\ell$, where we take $\hat{A}^+_{\sigma^k\omega}(c_k(\ell))$ for $\ell\leq 0.5$ and $\hat{A}^-_{\sigma^k\omega}(c_k(\ell))$ otherwise. The maximum $0.891$ occurs at $\ell^*=0.47$. The red section of the curve corresponds to $\hat{A}^+_{\sigma^k\omega}(c_k(\ell))$ and the blue section to $\hat{A}^-_{\sigma^k\omega}(c_k(\ell))$.}
    \label{fig:RhoMean}
\end{figure}

Figure \ref{fig:NonCS} shows the graph of
$f_2^{(20,10)}(\sigma^k\omega)$ with the threshold $c_k(\ell^*)$ for
$k=0,\ldots,5$, and in each case the set
$\hat{A}^+_{\sigma^k\omega}(c_k(\ell^*))$ is indicated by shading.
Since coherent sets are required to be connected, we must find the
interval closest to each $\hat{A}^+_{\sigma^k\omega}(c_k(\ell^*))$.
For $k=0,2,3,4,5$ the set $\hat{A}^+_{\sigma^k\omega}(c_k(\ell^*))$
is itself an interval and we set
$A_{\sigma^k\omega}=\hat{A}^+_{\sigma^k\omega}(c_k(\ell^*))$. The
set $\hat{A}^+_{\sigma\omega}(c_k(\ell^*))$ has two components,
$[0.12.0.58]$ and $[0.60,0.61]$, and so we set
$A_{\sigma\omega}=[0.12,0.59]$. Table \ref{tab:CoSets} lists the
coherent sets $A_{\sigma^k\omega}$ and the values of $\rho_\mu\left(
A_{\sigma^k\omega}, A_{\sigma^{k+1}\omega}\right)$ for
$k=0,\ldots,5$.
\begin{table}[htb]
    \centering
        \begin{tabular}{|r|cccccc|}
        \hline
$k$ & 0 & 1& 2& 3& 4& 5 \\
$\omega_k$ & 1 & 2 & 3 & 2 & 4 & 3 \\
$A_{\sigma^k\omega}$ & $[0.11,0.58]$ &  $[0.12,0.59]$ &  $[0.35,0.82]$ &  $[0.07,0.54]$ &  $[0.35,0.82]$ & $[0.35,0.82]$ \\
$\rho_k$ & $0.89$ & $0.87$ & $0.87$ & $0.96$ &  $0.90$ & $0.87$ \\
        \hline
        \end{tabular}
    \caption{Coherent sets $A_{\sigma^k\omega}$ and the values of $\rho_\mu\left( A_{\sigma^k\omega}, A_{\sigma^{k+1}\omega}\right)$ for $k=0,\ldots,5$.}
    \label{tab:CoSets}
\end{table}

It is interesting to compare the locations of the coherent sets with
their corresponding maps in the mapping cocycle, see Figure
\ref{fig:Squares}. \begin{figure}[htb]
    \centering
    \psfrag{T1}[B][br]{$T_{\omega}=T_1$}
    \psfrag{T2}[B][br]{$T_{\sigma\omega}=T_2$}
    \psfrag{T3}[B][br]{$T_{\sigma^2\omega}=T_3$}
    \psfrag{T4}[B][br]{$T_{\sigma^3\omega}=T_2$}
    \psfrag{T5}[B][br]{$T_{\sigma^4\omega}=T_4$}
    \psfrag{T6}[B][br]{$T_{\sigma^5\omega}=T_3$}
    \includegraphics[width=0.75\textwidth]{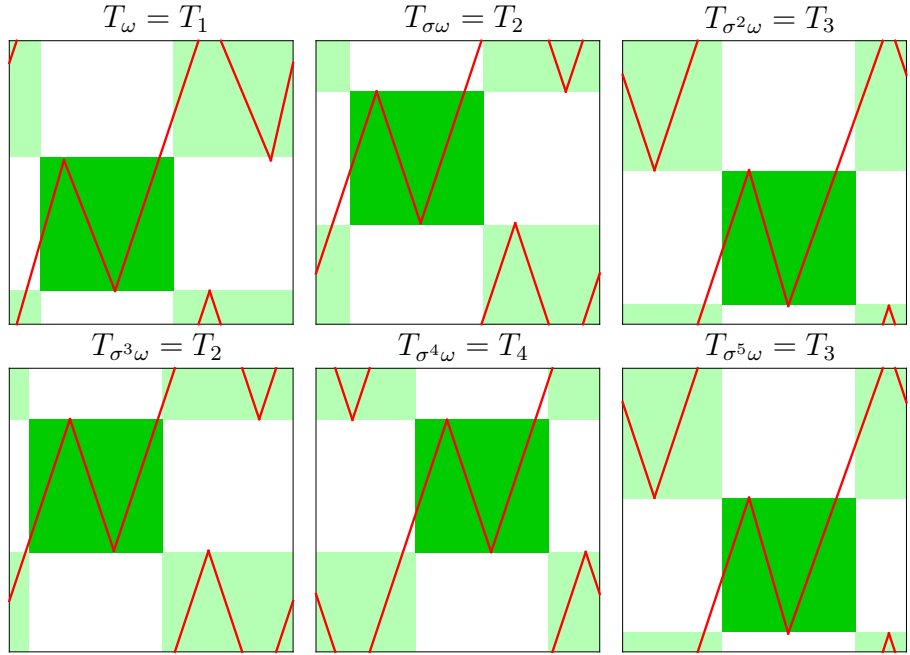}
    \caption{Graphs of $T_{\sigma^k\omega}$ showing $A_{\sigma^k\omega}\times A_{\sigma^{k+1}\omega}$ for $k=0,\ldots,5$.}
    \label{fig:Squares}
\end{figure}
As for the previously constructed family
$\set{R^{s(\omega,k)}([0.5,1])}_{k\in\N}$, the coherent sets
alternate between two positions separated by a rotation of
approximately $0.25$. However, the mean value of $\rho_\mu$ is
greater for the sequence $A_{\sigma^k\omega}$ constructed from Algorithm \ref{csseqalg} since in each case the
coherent set matches up well with local maxima and minima of the
preceding map.

\subsection{Coherent Sets in a 2D continuous time nonautonomous system}
\label{sec:ctsCS}

We apply Algorithm \ref{csalg} to the Oseledets subspaces
$W^{(80,40)}(z)$ and $W^{(80,40)}(\xi(10,z))$
calculated in Section \ref{5.5.1} and displayed in Figure
\ref{fig::chaoticpushforward}.
% For the perturbed
%system~\eqref{eq::travellingwavemixing} we use $M=80$, $N=40$, $t=0$
%in Algorithm \ref{ctsPalg} to compute $P^{(80)}(-40)$ and
%approximate $W_2(0)$ by the push forward as described in
%Algorithm~\ref{ctseigenmodealg}.
The optimal coherent sets $\hat{A}^+_{z}$ and
$\hat{A}^+_{\xi(10,z)}$ are obtained at the threshold values
$c^*=0.0043$ and $\eta(c^*)=0.0052$, which gives
$\rho_\mu(\hat{A}^+_{z}(c^*),\hat{A}^+_{\xi(10,z)}(\eta(c^*)))=0.9605$,
see Figure~\ref{fig::thresholdcurvechaotic3}.
\begin{figure}[htb]
\centerline{\includegraphics[scale=0.5]{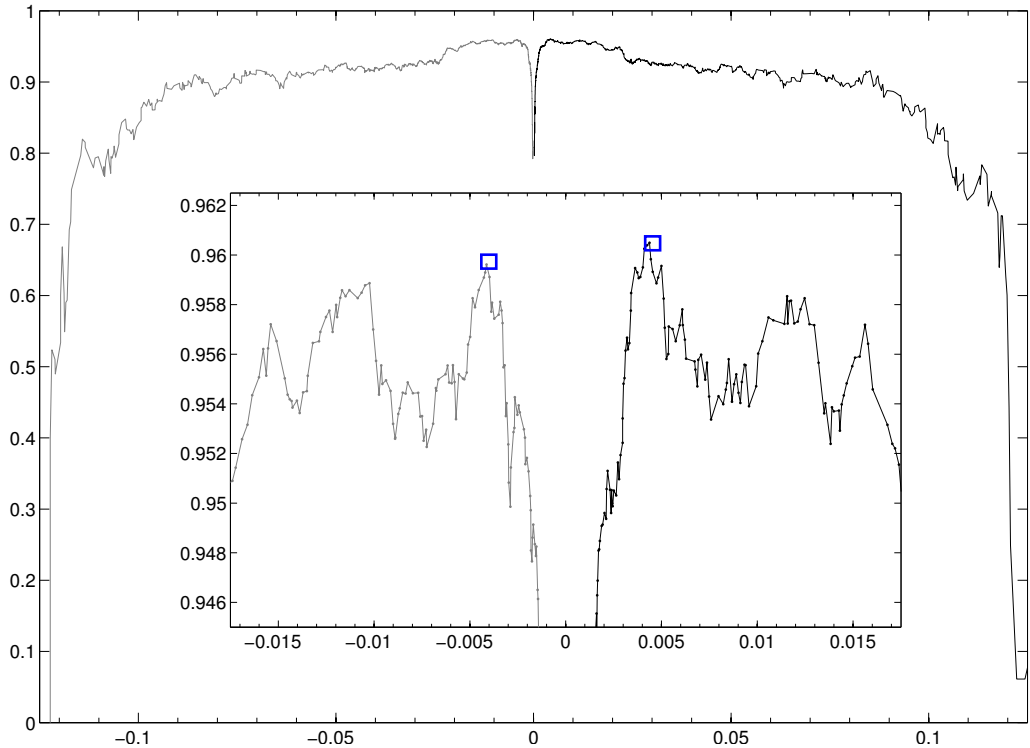}}
\caption{Thresholding curve
$\rho_\mu(\hat{A}^{-}_{z}(c),\hat{A}^{-}_{\xi(10,z)}(\eta(c)))$
and
($\rho_\mu(\hat{A}^{+}_{z}(c),\hat{A}^{+}_{\xi(10,z)}(\eta(c)))$
are plotted in grey and black, respectively. The optimal threshold is
marked with a rectangle.} \label{fig::thresholdcurvechaotic3}
\end{figure}
For $\hat{A}^-_{z}$ and $\hat{A}^-_{\xi(10,z)}$ the optimal
threshold is at $c^*=-0.0040$ and $\eta(c^*)=-0.0051$ and
$\rho_\mu(\hat{A}^-_{z}(c^*),\hat{A}^-_{\xi(10,z)}(\eta(c^*)))=0.9599$.
The coherent sets at $\hat{A}^\pm_{z}$ and $\hat{A}^\pm_{\xi(10,z)}$
and the images of sample points in $\hat{A}^\pm_{z}$ are shown in
%Figure~\ref{fig::eigenmodewavechaotic}
Figure~\ref{fig::coherentsetwavechaotic}.

%\begin{itemize}
%\item Insert all figures from previous section recomputed for this
%example.
%\item
%Let $F_0, F_\tau$ denote sets bounded by FTLE ridges at times 0,
%$\tau$ chosen so that the sets are closest to those identified by
%eigenfunctions.  Insert figure showing $\phi_{-\tau}(F_\tau,\tau)$,
%and overlay with border of $F_0$ to demonstrate the two sets are not
%as similar as $\phi_{-\tau}(A_\tau,\tau)$ and $A_0$.
%\item Perform
%comparison of $\rho$ calculation for the coherent family and sets
%defined by nearby FTLE ridges to demonstrate that FTLE ridges define
%inferior coherent families.
%\end{itemize}
\begin{figure}[htb]
\centerline{\includegraphics[scale=0.5]{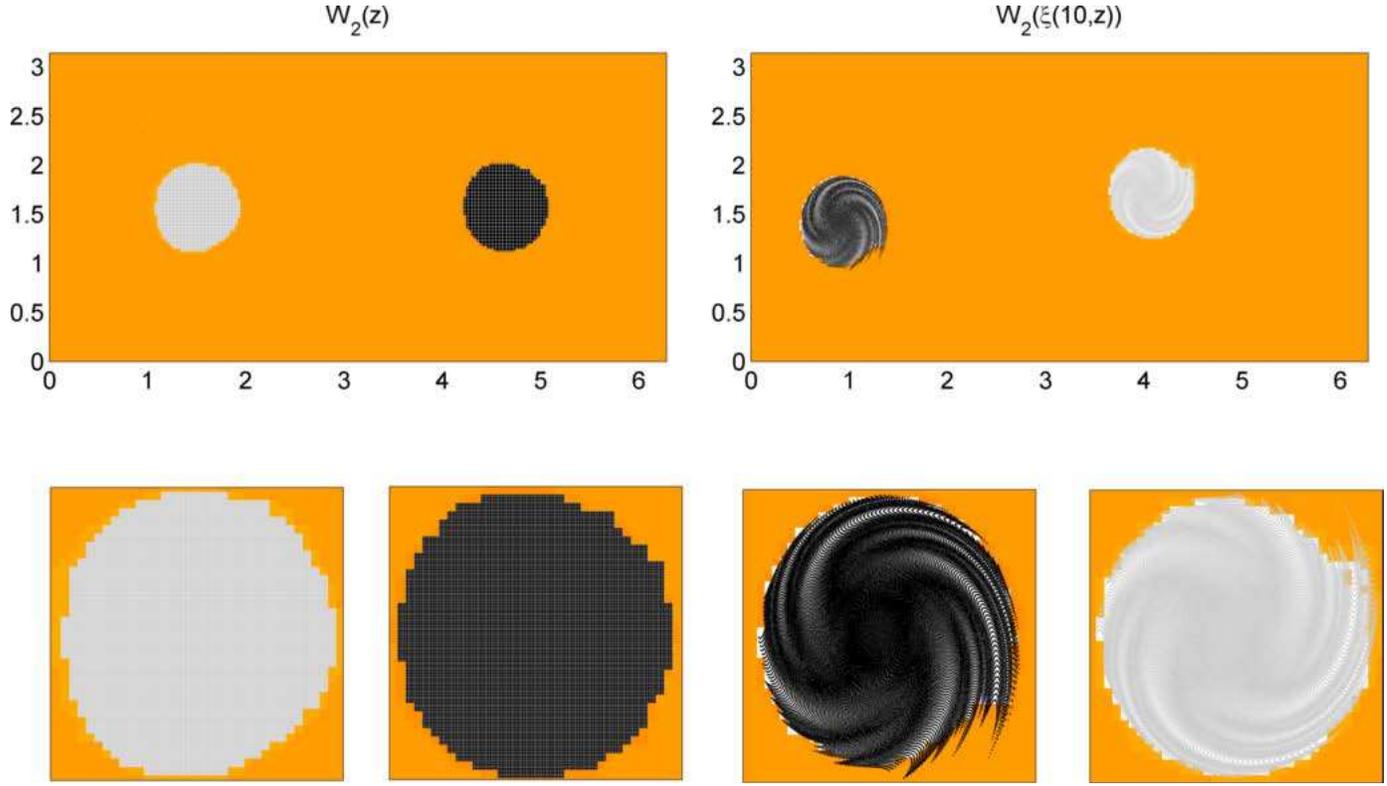}}
\caption{[Top left] The coherent sets $A^+_{z}$ (grey) and $A^-_{z}$
(black) and [Top right] $A^+_{\xi(10,z)}$ (grey) and $A^-_{\xi(10,z)}$ (black) are
identified by thresholding the Oseledets functions.
Overlays of $\phi(10,z,A^+_{z})$ (grey) and $\phi(10,z,A^-_{z})$ (black) are also shown. [Bottom left]
Zooms of $A^{+}(z)$  and $A^-(z)$ .
[Bottom right] Overlays of $\phi(10,z,A^{\pm}_z)$
(grey/black dots) on $A^{\pm}_{\xi(10,z)}$ (white), displaying the loss of
mass over 10 time units duration from $z=(0,1,1.5)$.}
\label{fig::coherentsetwavechaotic}
\end{figure}

%\begin{figure}[htb]
%\centerline{\includegraphics[scale=0.5]{wave12_coherentsets_zoomin_redo}}
%\caption{(left): Zooms of $A^{+}(0)$ (dark grey) and $A^-(0)$ (light
%grey). (right): Overlays of $\phi(10,0,A^{\pm}(0))$ (grey dots) on
%$A^{\pm}(10)$ (white), displaying the loss of mass over 10 time
%units duration from $t=0$.} \label{fig::wave12_coherentsets_zoomin}
%\end{figure}

In Figure \ref{fig::coherentsetwavechaotic} we note that the
grey set $A^+_z$ on the left at time $z=(0,1,1.5)$ flows approximately to the light grey
set $A^+_{\xi(10,z)}$ on the right at time $\xi(10,z)$.  Similarly for the black sets $A^-_z$ and $A^-_{\xi(10,z)}$.  This
carrying of the time $z$ coherent sets to the time $\xi(10,z)$ coherent sets by
the aperiodic flow is only approximate, as
$\rho_\mu(\hat{A}^+_{z}(c^*),\hat{A}^+_{\xi(10,z)}(\eta(c^*)))=0.9605$
and
$\rho_\mu(\hat{A}^-_{z}(c^*),\hat{A}^-_{\xi(10,z)}(\eta(c^*)))=0.9599$.
Thus, we expect a loss of about 10\% under the advection of the
flow. Figure \ref{fig::coherentsetwavechaotic} also zooms onto
$A^+_{z}$ and $A^-_{z}$ to demonstrate this loss of mass. To
make this loss even more apparent, we continue to flow forward for
50 time units.
\begin{figure}[htb]
\centerline{\includegraphics[scale=0.5]{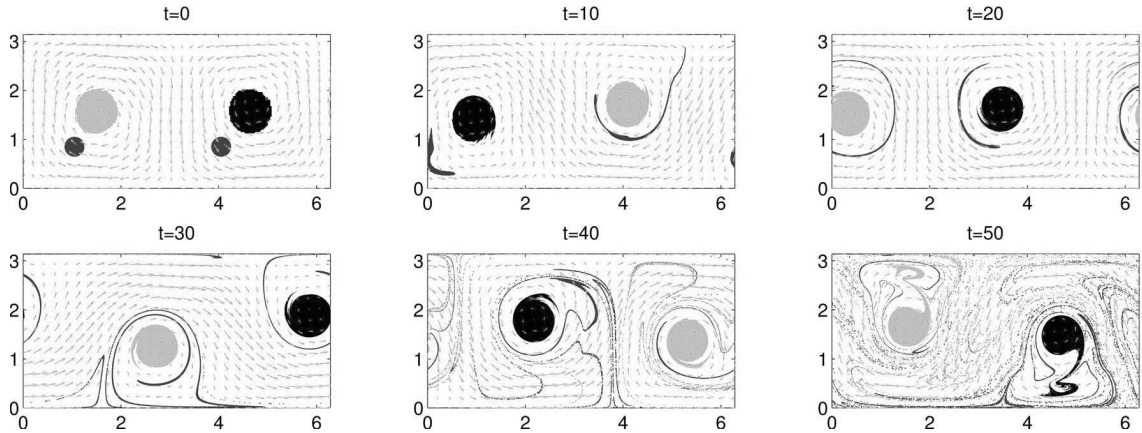}}
\caption{Trajectories of the perturbed
system~(\ref{eq::travellingwavemixing}) for $\varepsilon=1$. The large light grey ($A^+_{\xi(t,z)}$) and black ($A^-_{\xi(t,z)}$)
blobs are the coherent sets identified by our approach. The other
(medium grey) blobs are chosen nearby the coherent sets to show strong mixing away from the coherent regions.}\label{fig::wavechaoticmixing}
\end{figure}
Figure~\ref{fig::wavechaoticmixing} shows that the (black) coherent
sets $A^+_{z}$ and $A^-_{z}$ do indeed disperse over time,
however at a much slower rate than the arbitrarily chosen (grey)
sets.  The coherent sets $A^+_{z}, A^-_{z}$ are just
single elements of a time parameterised family \{$A^+_{\xi(t,z)},
A^-_{\xi(t,z)}\}_{t\ge 0}$ of coherent sets that at any given initial
time describe those sets that will disperse most slowly over a
duration of 10 time units.

\section{Final Remarks}
We have formulated a new mathematical and algorithmic approach for
identifying and tracking coherent sets in nonautonomous systems. Our
new approach generalises existing successful transfer operator
methodologies that have been used in the autonomous setting. Our
constructions address the question raised by \cite{liu_haller_04} of
how to study strange eigenmodes and persistent patterns observed in
forced fluid flows in the general time-dependent situation. Future
work will include applying these techniques to detect and track
mobile coherent regions in oceanic and atmospheric flows, extending significantly the flow times studied in \cite{froyland_padberg_england_treguier_07,npg} and \cite{SFM}.

\bibliographystyle{plain}

\bibliography{nonauto_18}

\end{document}